\documentclass[reqno]{article}
\usepackage[top=3cm, bottom=2.5cm, left=3cm, right=3cm]{geometry}
\usepackage{lmodern}

\usepackage{amssymb}
\usepackage{amsmath}
\usepackage{mathtools}
\usepackage{amsthm}
\usepackage{graphicx}
\usepackage{array}
\usepackage{titlesec}
\usepackage{eepic}
\usepackage{multicol}
\usepackage[usenames,dvipsnames]{xcolor}
\usepackage{color}
\usepackage[linktocpage=true]{hyperref}
\usepackage[hypcap=true]{caption}
\usepackage[cmtip,all]{xy}
\usepackage{enumitem}
\usepackage{leftidx}
\usepackage[dvipsnames]{xcolor}
\usepackage[mathscr]{euscript}
\usepackage{setspace}
\usepackage{needspace}

\usepackage{parskip}

\makeatletter
\def\thm@space@setup{%
  \thm@preskip=\parskip \thm@postskip=0pt
}
\makeatother

\allowdisplaybreaks

\usepackage{titlesec}
\titleformat{\section}[block]{\color{black}\large\bfseries\filcenter}{\thesection.}{0.5em}{}
\titleformat{\subsection}[hang]{\bfseries}{}{0.5em}{}

\numberwithin{equation}{section}

\newtheorem{theorem}{Theorem}[section]
\newtheorem{lemma}[theorem]{Lemma}
\newtheorem{proposition}[theorem]{Proposition}
\newtheorem{corollary}[theorem]{Corollary}
\newtheorem{definition}[theorem]{Definition}
\newtheorem{problem}[theorem]{Problem}

\newtheorem{remark}[theorem]{Remark}

\newtheorem{ltheorem}{Theorem}

\titleformat{\subsection}[runin]{\bfseries}{}{}{}[.]
\titleformat{\subsubsection}[runin]{\bfseries}{}{}{}[.]

\makeatletter
\renewenvironment{proof}[1][\proofname]{%
   \par\pushQED{\qed}\normalfont%
   \topsep6\p@\@plus6\p@\relax
   \trivlist\item[\hskip\labelsep\bfseries#1\@addpunct{.}]%
   \ignorespaces
}{%
   \popQED\endtrivlist\@endpefalse
}
\makeatother

\addtocounter{tocdepth}{0}

\newcommand{\CC}{\mathbf{C}}

\newcommand{\EE}{\mathbf{E}}

\newcommand{\PP}{\mathbf{P}}

\newcommand{\RR}{\mathbf{R}}
\newcommand{\TT}{\mathbf{T}}

\newcommand{\B}{\mathcal{B}}
\newcommand{\C}{\mathcal{C}}

\newcommand{\G}{\mathcal{G}}

\def\L{\mathcal{L}}
\newcommand{\M}{\mathcal{M}}
\newcommand{\N}{\mathcal{N}}
\newcommand{\Q}{\mathcal{Q}}
\newcommand{\R}{\mathcal{R}}

\newcommand{\U}{\mathcal{U}}

\newcommand{\X}{\mathcal{X}}

\newcommand{\Se}{\mathscr{S}}

\newcommand{\into}{\hookrightarrow}

\newcommand{\tto}{\longrightarrow}

\DeclareMathOperator*\wstspan{\overline{\mathrm{span}^{\wast}}}

\def\XXint#1#2#3{{\setbox0=\hbox{$#1{#2#3}{\int}$ }
\vcenter{\hbox{$#2#3$ }}\kern-.6\wd0}}

\def\1{\mathbf{1}}
\def\Id{\mathrm{id}}
\def\H{{H}}
\def\Q{\mathcal{Q}}
\def\M{\mathcal{M}}

\def\N{\mathcal{N}}
\def\Zent{\mathcal{Z}}
\def\Ad{\mathrm{Ad}}

\newcommand{\vertiii}[1]{
 {\left\vert\kern-0.25ex\left\vert\kern-0.25ex\left\vert #1 
  \right\vert\kern-0.25ex\right\vert\kern-0.25ex\right\vert}
}

\def\Prj{\mathscr{P}}
\def\Ball{\mathrm{Ball}}
\def\CB{\mathcal{C B}}

\def\wast{{{\mathrm{w}}^\ast}}
\def\op{\mathrm{op}}
\def\cb{\mathrm{cb}}
\def\Aut{\mathrm{Aut}}
\def\Tr{\mathrm{Tr}}

\def\Inn{\mathrm{Inn}}
\def\Out{\mathrm{Out}}

\def\algtensor{\otimes_{\mathrm{alg}}}

\def\weaktensor{\bar{\otimes}}

\title{
  Lower bounds in $L^p$-transference for crossed-products
}

\author{
  Adri\'an M. Gonz\'alez-P\'erez
  \thanks{
    The author has been partially funded by the ANR grant HASCON 
  } 
}
\date{}

\begin{document}

\maketitle

\begin{abstract}
  Let $\Gamma \curvearrowright \Omega$ be a measure-preserving action and $\L \Gamma \into L^\infty(\Omega) \rtimes \Gamma$ 
  the natural inclusion of the group von Neumann algebra into the crossed product. When $\mu(\Omega) = \infty$,
  we have that this natural embedding is not trace-preserving and therefore does not extends boundedly to the associated noncommutative $L^p$-spaces.
  Nevertheless, we show that when $\Omega$ has an invariant mean there is an isometric embedding of 
  $L^p(\L \Gamma)$ into an ultrapower of $L^p(\Omega \rtimes \Gamma)$ that intertwines Fourier multipliers 
  and it is $\L \Gamma$-bimodular. As a consequence we obtain the lower transference bound
  \[
    \big\| T_m: L^p(\L \Gamma) \to L^p(\L \Gamma) \big\|
    \leq
    \big\| (\Id \rtimes T_m): L^p(\Omega \rtimes \Gamma) \to L^p(\Omega \rtimes \Gamma) \big\|, 
  \]
  and the same follows for complete norms. The techniques employed are in line with those of \cite{Gon2018CP} in which the reverse inequality 
  was proven for measure preserving Zimmer-amenable actions.
  Both are preceded by the pioneering works of Neuwirth/Ricard and Caspers/de la Salle \cite{NeuRic2011, CasSall2015} for amenable groups. 
  
  The condition of having an invariant mean is quite restrictive. Therefore, we explore whether other equivariant embeddings
  $\Phi: \L \Gamma \to L^\infty(\Omega)$ yield a transference result as above. In this context, a map is considered equivariant 
  if it is co-multiplicative with respect to the canonical co-multiplication of $\L \Gamma$ and the canonical co-action of 
  $\L \Gamma$ on $L^\infty(\Omega) \rtimes \Gamma$. Those maps are given by linear extension of
  \[
    \Phi(\lambda_g) = \varphi_g \rtimes \lambda_g,
  \]
  for a collection of functions $(\varphi_g)_{g \in \Gamma} \subset L^\infty(\Omega)$. We start by noticing that the multiplicative,
  completely positive and completely bounded equivariant maps can be easily characterized.
  In particular, completely positive equivariant maps are given by the matrix coefficients
  of unitary $1$-cocycles $\kappa: \Gamma \times \Omega \to \U(H)$.
  Then, we show that the transference proof above works verbatim whenever $\Phi$ is completely positive, amenable in the sense of Popa
  and Anantharaman-Delaroche \cite{Anan1995AmenableCorr} and intertwines Fourier multipliers at the $L^2$-level. Although no new transference results are obtained, 
  both the classification of equivariant maps and the study their amenability may be of independent interest to some readers.
\end{abstract}

\section*{Introduction}
\textbf{Fourier and Herz-Schur multipliers.}
Let $\Gamma$ be a discrete group and let $\L \Gamma \subset \B(\ell^2 \Gamma)$ be its left regular von Neumann algebra, that is, the weak-$\ast$ closure of the group algebra $\CC[ \Gamma ]$ under the left regular representation $\lambda: \Gamma \to \B(\ell^2 \Gamma)$. Let $m \in \ell^\infty(\Gamma)$ be a function. The (potentially unbounded) operator $T_m: \CC[\Gamma] \subset \L \Gamma \to \L \Gamma$ given by
\[
  T_m \Big( \sum_{g \in \Gamma} a_g \lambda_g \Big)
  \, = \,
  \sum_{g \in \Gamma} a_g \, m(g) \, \lambda_g,
\] 
is called the \emph{Fourier multiplier of symbol $m$}. Whenever $T_m: \L \Gamma \to \L \Gamma$ is bounded/completely bounded, it is said that $m$ is a bounded/completely bounded Fourier multiplier. Observe that, if $\iota:\L \Gamma \into \B(\ell^2 \Gamma)$ is the natural embedding, we have that the following diagram commutes
\begin{equation*}
\xymatrix@C=2cm{
  \L \Gamma \ar[r]^{\iota} \ar[d]^{T_m} & \B(\ell^2 \Gamma) \ar[d]^{H_m}\\
  \L \Gamma \ar[r]^{\iota} & \B(\ell^2 \Gamma),
}
\end{equation*}
where $H_m: M_{\Gamma \times \Gamma}(\CC) \subset \B(\ell^2 \Gamma) \to \B(\ell^2 \Gamma)$ is the so-called \emph{Herz-Schur multiplier of symbol $m$}, that is the (again potentially unbounded) operator given by
\[
  H_m \Big( \sum_{g \in \Gamma} a_{g,h} \, e_{g, h} \Big)
  \, = \,
  \sum_{g \in \Gamma} a_{g, h} \, m(g \, h^{-1}) \, e_{g, h}.
\]
Since $\iota$ is isometric, we have that the norm of $T_m: \L \Gamma \to \L \Gamma$ is bounded by that of $H_m: \B(\ell^2 \Gamma) \to \B(\ell^2 \Gamma)$. The reciprocal was shown to be true by Bozejko/Fendler \cite{BozejkoFendler1984HSFourier} but in that case the complete boundedness of $T_m$ is required. In sum, its is known that for every $m$, the following holds
\begin{enumerate}[leftmargin=1.25cm, label={\textbf{(\Alph*)}}, ref={(\Alph*)}]
  \item \label{itm:IntroA}
  $\displaystyle{\big\| H_m: \B(\ell^2 \Gamma) \to \B(\ell^2 \Gamma) \big\| \leq \big\| T_m: \L \Gamma \to \L \Gamma \big\|_\cb}$.
  \item \label{itm:IntroB}
  $\big\| H_m: \B(\ell^2 \Gamma) \to \B(\ell^2 \Gamma) \big\|_\cb = \big\| H_m: \B(\ell^2 \Gamma) \to \B(\ell^2 \Gamma) \big\|$, ie: $H_m$ is automatically cb. 
  \item \label{itm:IntroC}
  $\displaystyle{\big\| T_m: \L \Gamma \to \L \Gamma \big\| \leq \big\| H_m: \B(\ell^2 \Gamma) \to \B(\ell^2 \Gamma) \big\|}$.
\end{enumerate}
The fact that $H_m$ is automatically continuous holds for general \emph{Schur multipliers}, ie bounded operators given by $e_{i, j} \mapsto m_{i,j} \, e_{i, j}$ and more general for bimodular operators, see \cite{Smith1991}. For Fourier multipliers it is known that boundedness and complete boundedness are not equivalent, since there are examples of multipliers that fail to be completely bounded over nonamenable groups, see \cite{HaagKra1994, HaagSteenSw2009}. 

\textbf{Noncommutative $L^p$-bounds and transference.}
Given a semifinite von Neumann algebra $\M$ and a normal, semifinite and faithful trace $\tau: \M_+ \to [0, \infty]$ their noncommutative $L^p$-spaces, see \cite{Terp1981lp, PiXu2003}, can be defined as the spaces of $\tau$-measurable operators $x$ such that
\[
  L^p(\M, \tau) \, = \,  \big\{ x : \| x \|_p \, := \, \tau \big( |x|^p \big)^\frac1{p} \, < \, \infty \big\}.
\]
In the case of $\B(\ell^2)$ with its usual trace $\Tr$ the associated noncommutative $L^p$-spaces are called the $p$-Schatten classes and denoted by $S^p(\ell^2)$. The group von Neumann algebra $\L \Gamma$ of a discrete group $\Gamma$ admits a normal tracial state $\tau: \L \Gamma \to \CC$, given by $\tau(x) = \langle \delta_e, x \, \delta_e \rangle$. We will denote their associated $L^p$-spaces by $L^p(\L \Gamma)$. 
Asking whether, given a symbol $m$, their associate Fourier and Herz-Schur multipliers are bounded over $L^p(\L \Gamma)$ and $S^p(\ell^2 \Gamma)$ respectively its an extremely difficult question that has received attention for its connections with approximations properties \cite{LaffSall2011}, the theory of Markovian semigroups \cite{GonJunPar2015, JunMeiPar2014Riesz} as well as the convergence of Fourier series over non-Abelian groups \cite{HongWangWang2020} among other problems. For noncommutative $L^p$-spaces, the analogues of points \ref{itm:IntroB} and \ref{itm:IntroC} are widely open. Nevertheless, point \ref{itm:IntroA} was shown to be true by Neuwirth/Ricard \cite{NeuRic2011}, which also showed that
\begin{enumerate}[leftmargin=1.25cm, label={\textbf{(\Alph*')}}]
  \item \label{itm:IntroC'}
  $\displaystyle{
    \big\| T_m: L^p(\L \Gamma) \to L^p(\L \Gamma) \big\| 
    \, \leq \, 
    \big\| H_m: S^p(\ell^2 \Gamma) \to S^p(\ell^2 \Gamma) \big\|,
  }$
  when $\Gamma$ is amenable.
\end{enumerate}
The same follows from complete norms. Their technique, which is sometimes called \emph{transference}, works by using the amenability of $\Gamma$ to construct, for every $1 \leq p < \infty$, a completely isometric embedding
\[
  L^p(\L \Gamma) \xrightarrow{\quad J_p \quad} \prod_{\U} S^p(\ell^2 \Gamma),
\]
where the space in the right hand side is a proper ultrapower of the Schatten classes, that intertwines the operators $T_m$ and $H_m^\U$. This technique was later generalized to locally compact groups \cite{CasSall2015} and to the context of trace-preserving Zimmer-amenable actions $\theta:\Gamma \to \Aut(\Omega, \mu)$ on a semifinite measure space \cite{Gon2018CP}. Indeed, in \cite{Gon2018CP} a (complete) isometry
\[
  L^p(\Omega \rtimes_\theta \Gamma) \xrightarrow{\quad J_p \quad} \prod_{\U} L^p(\Omega) \otimes_p S^p(\ell^2 \Gamma),
\]
that intertwines $\Id \rtimes T_m$ and $(\Id \otimes H_m)^\U$ was constructed using the amenability of $\theta$. That construction implies that 
\begin{equation}
  \label{eq:IntroOneSide}
  \begin{split}
    \big\| \Id \rtimes T_m: & L^p(\Omega \rtimes_\theta \Gamma) \to L^p(\Omega \rtimes_\theta \Gamma) \big\|\\
    & \, \leq \,
    \big\| \Id \otimes H_m: L^p(\Omega) \otimes_p S^p(\ell^2 \Gamma) \to L^p(\Omega) \otimes_p S^p(\ell^2 \Gamma) \big\|,
  \end{split}
\end{equation}
and the same follows for complete bounds. Although, Zimmer-amenable actions preserving a finite measure (or more generally a mean) can only be constructed for amenable groups, there are plenty of examples of non-amenable groups which act in a Zimmer-amenable way on a semifinite measure space. Indeed, any exact discrete group admits such an action, see the comments after Remark \ref{rmk:Several}. 

If the reverse inequality of \eqref{eq:IntroOneSide} were true for some action $\theta$, ie if
\begin{equation}
  \label{eq:ReverseIneq}
  \big\| T_m: L^p(\L \Gamma) \to L^p(\L \Gamma) \big\|_\cb
  \, \leq \,
  \big\| \Id \rtimes T_m: L^p(\L \Gamma) \to L^p(\L \Gamma) \big\|_\cb
\end{equation}
holds, then using the fact that the complete norm of $H_m$ is always bounded by that of $T_m$, we will have that the complete norms of $T_m: L^p(\L \Gamma) \to L^p(\L \Gamma)$ and $H_m: S^p(\ell^2 \Gamma) \to S^p(\ell^2 \Gamma)$ coincide. Sadly, we have only been able to obtain the reverse inequality when the action admits an invariant mean. Indeed, it holds that

\begin{ltheorem}
  \label{thm:IsometricInclusion}
  Let $\1 \leq p \leq \infty$ and $\U$ be a proper ultrafilter.
  If $\Omega$ has a $\theta$-invariant mean, there is a complete isometry
  \begin{equation*}
    L^p(\L \Gamma) \xrightarrow{\quad J_p \quad} \prod_{\U} L^p(\Omega \rtimes_\theta \Gamma)
  \end{equation*}
  which satisfies that
  \begin{equation}
    \xymatrix@C=3cm{
      L^p(\L \Gamma) \ar[d]^-{T_m} \ar[r]^-{J_p}
        & \displaystyle{ \prod_{\U} L^p(\Omega \rtimes_\theta \Gamma) \ar[d]^{(\Id \rtimes T_m)^\U}}\\
      L^p(\L \Gamma) \ar[r]^-{J_p}
        & \displaystyle{ \prod_{\U} L^p(\Omega \rtimes_\theta \Gamma) }\\
    }
    \label{dia:Intertwining}
  \end{equation}
  and is $\L \Gamma$-bimodular.
  
  As a consequence we have that, for every $1 \leq p \leq \infty$,
  \begin{equation}
    \label{eq:lowerTransferenceBnd}
    \big\| T_m: L^p(\L \Gamma) \to L^p(\L \Gamma) \big\|
    \leq
    \big\| (\Id \rtimes T_m): L^p(\Omega \rtimes_\theta \Gamma) \to L^p(\Omega \rtimes_\theta \Gamma) \big\|,    
  \end{equation}
  the same follows for completely bounded norms.
\end{ltheorem}

This result does not provide new examples of groups for which both the Fourier and Herz-Schur multipliers have the same norm in $L^p$ since every group $\Gamma$ admitting a Zimmer-amenable action with an invariant mean is amenable. 

In the search for examples beyond amenable groups we explore if changing the natural embedding $\L \Gamma \into L^\infty(\Omega) \rtimes_\theta \Gamma$ by other completely positive maps $\Phi: \L \Gamma \to L^\infty(\Omega) \rtimes_\theta \Gamma$ will give weaker conditions than the existence of an invariant mean. Since we want to intertwine the operators $T_m$ and $\Id \rtimes T_m$ the map $\Phi$ must be of the form
\[
  \lambda_g \xmapsto{ \quad \Phi \quad } \varphi_g \rtimes \lambda_g,
\]
for some collection of functions $(\varphi_g)_{g \in \Gamma}$. Such maps, which we will call \emph{equivariant}, can be easily classified as follows

\begin{ltheorem}
  \label{thm:EquivariantMap}
  Let $\Phi:\L \Gamma \to  L^\infty(\Omega) \rtimes_\theta \Gamma$ be a normal and equivariant map of symbol $(\varphi_g)_{g \in \Gamma}$ then
  \begin{enumerate}[leftmargin=1.25cm, label={\rm \textbf{(\roman*)}}, ref={\rm {(\roman*)}}]
    \item \label{itm:EquivariantMap1}
    $\Phi$ is a $\ast$-homomorphism iff $\varphi: \Gamma \to L^\infty(\Omega; \TT)$ is a multiplicative $1$-cocycle.
    \item \label{itm:EquivariantMap2}
    $\Phi$ is unital and completely positive iff there exists a Hilbert space $H$ and a multiplicative $1$-cocycle $\kappa: \Gamma \to L^\infty(\Omega; \, H)$ and a unit vector $\xi \in H$ such that
    \[
      \varphi_g(\omega) = \langle \xi, \kappa_g(\omega) \, \xi \rangle
    \] 
    \item \label{itm:EquivariantMap3}
    $\Phi$ is completely bounded iff there is a Hilbert space $H$ and maps $\Xi$, $\Sigma \in L^\infty(\Omega \times \Gamma; H)$ such that 
    \begin{equation}
      \label{eq:Decomposition}
      \varphi_{g \, h^{-1}}(\omega)
      \, = \, \big\langle \Xi_g, \Sigma_h \big\rangle(\theta_{g} \, \omega)
      \, = \, \big\langle \Xi_g(\theta_{g} \, \omega), \Sigma_h(\theta_{g} \, \omega) \big\rangle.
    \end{equation}
    Furthermore it holds that
    \[
      \big\| \Phi: \L \Gamma \to L^\infty(\Omega) \rtimes_\theta \Gamma \big\|_\cb
      \, = \, 
      \inf \Big\{ \| \Sigma \|_{L^\infty(\Omega \times \Gamma; H)} \, \| \Xi \|_{L^\infty(\Omega \times \Gamma; H)} \Big\}
    \]
    where the infimum is taken over all the $\Sigma$ and $\Xi$ as in \eqref{eq:Decomposition}.
  \end{enumerate}
\end{ltheorem}

Observe that the result above is a straightforward generalization of \cite{Jolissaint1992Multipliers}, in the case of \ref{itm:EquivariantMap3} as well as the classical correspondence between positive type functions and group representations, see \cite[Section 3.3]{Foll1995} and \cite[Appendix C]{BeHarVal2008}. Although the theorem above is not surprising it seems that it is not in the literature, therefore we have chosen to include it in its full generality, although we will just use the points \ref{itm:EquivariantMap1} and \ref{itm:EquivariantMap2}. 

\textbf{Weak containment of correspondences and transference.}
In Sections \ref{sct:AmenableHom} and \ref{sct:AmenableCP} we will connect the transference technique described above with the theory of weak containment of correspondences over von Neumann algebras. Given two von Neumann algebras $\N$ and $\M$, an \emph{$\N$-$\M$-correspondence} is a Hilbert space $H$ with two normal and commuting representation that turn $H$ into an $\N$-$\M$-bimodule. Those objects form a natural category with the bounded bimodular maps as their morphisms, see \cite{Popa1986Notes}, \cite[Appendix B]{ConnesBook1994} or \cite[Chapter 13]{AnaPopa2017II1}. This category behaves in the context of von Neumann algebras in a way that is reminiscent of the category of representations in the context of groups, see \cite{ConnesJones1985}. Following this common analogy, the trivial $\N$-$\N$-bimodule $L^2(\N)$ is understood as an analogue of the trivial group representation $\1: \Gamma \to \CC$. Correspondences admit a natural notion of weak containment and a $\N$-$\M$ correspondence $H$ is said to be \emph{(left) amenable}
iff $H \weaktensor_\M \overline{H}$ weakly contains the trivial bimodule, see \cite{Anan1995AmenableCorr}. This notion generalizes the definition of amenable representation introduced in \cite{Bekka1990} in which $\rho$ is amenable iff $\1 \prec \rho \otimes \overline{\rho}$. Using a common GNS-type construction, we can associate to each completely positive map $\Phi: \N \to \M$ a cannonical $\N$-$\M$ correspondence $H(\Phi)$. The map $\Phi$ is said to be left amenable precisely when $H(\Phi)$ is. The key observation on Sections \ref{sct:AmenableHom} and \ref{sct:AmenableCP} is that the transference techniques used in \cite{NeuRic2011, CasParPerrRic2014, CasParPerrRic2014, Gon2018CP} can be understood as an extrapolation result by which the existence of a weak containment expressed in terms of Hilbert bimodules can be extended to noncommutative $L^p$-bimodules, see Theorem \ref{thm:Extrpolation}.

In order to explain this idea, in the particular case of $\ast$-homomorphism, let us denote $\N = \L \Gamma$. We need to recall that if $\pi: \N \to \M$ is some normal $\ast$-homomorphism into a semifinite von Neumann algebra $\M$, we have that $\pi$ is left amenable in the sense of \cite{Anan1995AmenableCorr} iff the trivial $\N$-bimodule $L^2(\N)$ is weakly contained in the $\N$-bimodule $L^2(\M)$, where the left and right actions are given by multiplication
\[
  x \cdot \xi \cdot y 
  \, = \,
  \pi(x) \, \xi \, \pi(y).
\]
But $L^2(\N) \prec L^2(\M)$ iff there is a $\N$-bimodular isometry
\[
  L^2(\N) \xrightarrow{\quad J \quad} \prod_\U L^2(\M)^\U,
\]
where the space on the right hand side is the bimodule given by a proper ultrapower. The insight is that this bimodular map at the $L^2$-level exists iff there is a $\N$-bimodular isometry $J_p: L^p(\N) \to L^p(\M)$. Similarly, if $J$ satisfy an intertwining identity between multiplier operators, so does $J_p$. Thus, transference theorem for $L^p$-spaces can be seen as extrapolation theorems in which an inclusion of Hilbert $\N$-bimodules is generalize to an inclusion of $L^p$ bimodules. Then, in Theorem \ref{thm:AmenableEquivariant} and Theorem \ref{thm:Amenable1CocyclesTrans} we classify which equivariant $\ast$-homomorphisms and completely positive maps respectively are amenable and admit an isometry $J$ intertwining the required multiplier operators. Our quest to find new examples beyond the need for invariant mean fails, since we obtain conditions that imply the existence of such a mean. Nevertheless, we consider that the resulting theorems may still be of interest.

An important question left open by the approach of this paper is the necessity of the $\N$-bimodularity on the extrapolation argument. In principle, just the existence of an intertwining $L^p$-isometry will suffice to prove a transference theorem. Nevertheless, it seems that all the current proofs in the literature use a certain bimodularity in order to prove the $L^p$-isometric character of $J_p$ and it is still open whether this condition is necessary or not.

\section{Equivariant maps into crossed products}
Let $\Gamma$ be a discrete countable group acting on a $\sigma$-finite measure space $(\Omega, \mu)$ by measure-preserving transformations. Denote the action by $\theta: \Gamma \to \Aut(\Omega, \mu)$. Recall that the \emph{group von Neumann algebra of $\Gamma$}, denoted by $\L \Gamma \subset \B(\ell^2 \Gamma)$ is the von Neumann algebra given by
\[
  \{\lambda_g : g \in \Gamma \}'' = \wstspan \{ \lambda_g : g \in \Gamma \} \subset \B(\ell^2 \Gamma),
\]
where $\lambda: \Gamma \to \B(\ell^2 \Gamma)$ is the left regular representation given by $\lambda_g(\delta_h) = \delta_{g \, h}$, where $g, h \in \Gamma$ and $(\delta_h)_{h \in \Gamma}$ is the canonical orthonormal base of placeholder functions in $\ell^2 \Gamma$.

The \emph{crossed-product} von Neumann algebra $L^\infty(\Omega) \rtimes_\theta \Gamma$, also called the \emph{group measure space} construction of the action $\theta: \Gamma \to \Aut(\Omega, \mu)$, is given the von Neumann algebra
\[
  L^\infty(\Omega) \rtimes_\theta \Gamma \subset \B(L^2 \Omega \otimes_2 \ell^2 \Gamma)
\]
generated by the representation $\1 \otimes \lambda: \Gamma \to \B(L^2 \Omega \otimes_2 \ell^2 \Gamma)$ and the $\ast$-homomorphism $\pi: L^\infty(\Omega) \to \B(L^2 \Omega \otimes_2 \ell^2 \Gamma)$ 
\[
  \pi(f) 
  \, = \,
  \sum_{h \in \Gamma} \theta_{h^{-1}}(f) \otimes e_{h, h},
\]
where $\theta_g(f)(\omega) = f(\theta_{g^{-1}} \omega)$. Observe that the algebra $L^\infty(\Omega) \rtimes_\theta \Gamma$ is given by by the weak-$\ast$ closure of finite sums of the form
\[
  x = \sum_{g \in \Gamma} f_g \rtimes \lambda_g,
\]
where $f \rtimes \lambda_g$ is just shorthand notation for the operator $\pi(f) \cdot (\1 \otimes \lambda_g)$.

Let $\tau: \L \Gamma \to \CC$ be the canonical trace of $\L \Gamma$, given by the vector state $\tau_\Gamma(x) = \langle \delta_e, x \delta_e \rangle$, that we will denote simply as $\tau$ when no ambiguity is present. The map $\EE: L^\infty(\Omega) \rtimes \Gamma \to L^\infty(\Omega)$, given by restriction to $L^\infty(\Omega) \rtimes_\theta \Gamma$ of $\Id \otimes \tau$ is a normal conditional expectation. A straightforward calculation gives that, over finite sums, it takes the form
\[
  \EE \Big( \sum_{g \in \Gamma} f_g \rtimes \lambda_g \Big) = f_e
\]
Furthermore it satisfies the following properties
\begin{itemize}
  \item It is \emph{faithful}, ie for every $x \geq 0$, $\EE[x] = 0$ iff $x = 0$.
  \item It is \emph{equivariant}, ie $\displaystyle{\EE \big[\lambda_g \, x \, \lambda_g^\ast \big] = \theta_g \EE[x]}$.
\end{itemize}
See \cite[Proposition 4.1.9]{BroO2008} for both facts above. Using those properties we obtain that every measure $\nu:L^\infty(\Omega)_+ \to [0,\infty]$ gives a tracial weight $\tau_\nu = \nu \rtimes \tau_\Gamma: (L^\infty(\Omega) \rtimes \Gamma)_+ \to [0,\infty]$ given by $\tau_\nu = \nu \circ \EE$, when $\nu$ is $\theta$-invariant. We will denote the canonical tracial weight associated with $\mu$ by $\tau_\rtimes$. The noncommutative $L^p$-spaces associated to both $(\L \Gamma, \tau)$ and $(L^\infty \Omega \rtimes_\theta \Gamma, \tau_\rtimes)$ would be denoted by $L^p(\L \Gamma)$ and $L^p(\Omega \rtimes_\theta \Gamma)$ respectively.

Observe also that both $\L \Gamma$ and $L^\infty(\Omega) \rtimes_\theta \Gamma$ admit the following natural comultiplication and coaction maps.

\begin{proposition}
  \label{prp:Coactions}
  There are normal $\ast$-homomorphisms
  \begin{align}
    \Delta:& \, \L \Gamma \tto \L \Gamma \weaktensor \L \Gamma,\\
    \Delta_\rtimes:& \, L^\infty(\Omega) \rtimes_\theta \Gamma
      \tto \big( L^\infty(\Omega) \rtimes_\theta \Gamma \big) \weaktensor \L \Gamma,
  \end{align}
  given by extension of $\lambda_g \mapsto \lambda_g \otimes \lambda_g$ and $f \rtimes \lambda_g \mapsto (f \rtimes \lambda_g) \otimes \lambda_g$ respectively. Those maps satisfy the following coassociativity properties
  \[
    (\Delta \otimes \Id) \circ \Delta = (\Id \otimes \Delta) \circ \Delta
    \quad \mbox{ and } \quad
    (\Delta_\rtimes \otimes \Id) \circ \Delta = (\Id \otimes \Delta) \circ \Delta.
  \]
\end{proposition}

The proposition above follows easily from the Fell absorption principle for representations and the Fell absorption principle for actions, see \cite[Propositions 4.1.7]{BroO2008}. We will use a form of the Fell absorption principle in the proof of Theorem \ref{thm:EquivariantMap} \ref{itm:EquivariantMap1}. Notice that in the case of $\Gamma$ Abelian, the map $\Delta$ is just the pullback over functions of the multiplication of $\widehat{\Gamma}$, the Pontryagin dual of $\Gamma$, while the map $\Delta_\rtimes$ is the pullback of the dual action $\widehat{\theta}: \widehat{\Gamma} \to \Aut(L^\infty \Omega \rtimes_\theta \Gamma)$, given by sending $\chi \in \widehat{\Gamma}$ into $f \rtimes \lambda_g \mapsto \langle \chi, g \rangle \, f \rtimes \lambda_g$. 

A map $T: \L \Gamma \to \L \Gamma$ is $\Delta$-equivariant $(\Id \otimes \Delta) \circ T = (\Id \otimes T) \circ \Delta$ iff it is a \emph{Fourier multiplier}, ie a map given by extension of
\[
  \lambda_g \longmapsto m(g) \, \lambda_g
\]
for some function $m \in \ell^\infty(\Gamma)$. That operator is called the \emph{Fourier multiplier of symbol $m$} and denote $T = T_m$, is we want to make the dependence on the symbol explicit. A map $\Phi:\L \Gamma \to L^\infty(\Omega) \rtimes_\theta \Gamma$ is equivariant iff $(\Id \otimes \Delta) \circ T = (T \otimes \Id) \circ \Delta_\rtimes = (\Id \otimes T) \circ \Delta_\rtimes$. Those operators are given by linear extension of the map 
\[
  \Phi(\lambda_g)
  \, = \,
  \varphi_g \rtimes \lambda_g, 
\]
for some symbol $\varphi: \Gamma \to L^\infty(\Omega)$. 

Let $\G$ be a topological group. The map $\kappa: \Gamma \to L^\infty(\Omega;\G)$ is a multiplicative $1$-cocycle, or simply a $1$-cocycle if it satisfies that
\[
  \kappa_{g \, h} = \kappa_{g} \, \theta_g(\kappa_h),
\]
where $\theta_g(\kappa)(\omega) = \kappa(\theta_g^{-1} \omega)$. We will say that $\kappa$ is a unitary $1$-cocycle if $\G = \U(H)$ for some Hilbert space $H$.

The following lemma is a trivial application of the embedding $\L \Gamma \into L^\infty(\Omega) \weaktensor \B(\ell^2 \Gamma)$.

\begin{lemma}
  \label{lem:CP}
  A equivariant map $\Phi: \L \Gamma \to L^\infty(\Omega) \rtimes_\theta \Gamma$, given by $\lambda_g \mapsto \varphi_g \rtimes \lambda_g$, is positivity preserving iff for every finite set $S = \{g_1, g_2, ..., g_r\} \subset \Gamma$ the matrices
  \[
    \big[ \theta_{g_i^{-1}}(\varphi_{g_i \, g_j^{-1}}) \big]_{i,j} \in  L^\infty(\Omega) \otimes M_S(\CC) .
  \]
  are positive definite. In that case, the map $\Phi$ is also completely positive.
\end{lemma}

\begin{proof}
  Observe that we have that the following diagram commutes
  \begin{center}~
    \xymatrix@C=1.5cm{
      \L \Gamma \ar[rr]^{\Phi} \ar[d] & & L^\infty(\Omega) \rtimes_\theta \Gamma \ar[d]\\
      \B(\ell^2 \Gamma) \ar[rd]_{\1_\Omega \otimes \Id} & & L^\infty(\Omega) \weaktensor \B(\ell^2 \Gamma)\\
      & L^\infty(\Omega) \weaktensor \B(\ell^2 \Gamma) \ar[ru]_{H_\Phi} &
    }
  \end{center}
  where the vertical arrows are the natural inclusions and the map $H_\Phi$ is the $L^\infty(\Omega)$-valued Schur multiplier given by
  \begin{equation}
    \label{eq:EmbeddingMatrix}
    \sum_{g,k \in \Gamma} f_{g, k} \otimes e_{g, k}
    \, \mapsto \,
    \sum_{g,k \in \Gamma} f_{g, k} \, \theta_{g^{-1}}(\varphi_{g \, k^{-1}}) \otimes e_{g, k}.
  \end{equation}
  Using that Schur multipliers preserve positivity iff they are positive definite gives the result.  
\end{proof}

We will need the following technical lemma which extends Grothendieck/Haagerup decomposition of bounded Schur multipliers to the $L^\infty(\Omega)$-valued case, see \cite{BozejkoFendler1984HSFourier}. Although the lemma is straightforward we include its proof for the sake of completeness.

\begin{lemma}
  \label{lem:LinftyValuedSchurs}
  Let $(\Omega, \mu)$ be a $\sigma$-finite measure space as before and $S$ be a countable discrete set. Let $m \in L^\infty(S \times S \times \Omega)$ be a function and 
  \[
    \B(\ell^2 S) \xrightarrow{\quad H_m \quad} L^\infty(\Omega) \weaktensor \B(\ell^2 S)
  \]
  the normal and bounded operator given as the $L^\infty(\Omega)$-valued Schur multiplier of symbol $m$
  \[
    H_m \Big( \sum_{s,t \in S} a_{s,t} \, e_{s,t} \Big)
    \, = \, 
    \sum_{s,t \in S} m_{s,t}(\omega) \, a_{s,t} \otimes e_{s,t}.
  \]  
  Then, there exists a Hilbert space $H$ and $\Xi, \Sigma \in L^\infty(\Omega \times S;H)$ such that 
  \[
    m_{s,t}(\omega) = \big\langle \Xi_s(\omega), \Sigma_t(\omega) \big\rangle
  \]
  and 
  \[
    \| H_m \| = \inf \Big\{ \| \Xi \|_{L^\infty(\Omega;H)} \, \| \Sigma \|_{L^\infty(\Omega;H)} \Big\},
  \]
  Reciprocally, all operators $H_m$ with the above decomposition are bounded.
\end{lemma}

\begin{proof}
  First observe that if $(\Omega, \mu)$ is a countable discrete space, so that $L^\infty(\Omega) \cong \ell^\infty(\Omega)$, the result above follows easily from the scalar case, see \cite{BozejkoFendler1984HSFourier}. Indeed, using that the restriction of $H_m$ to $\CC \delta_{\omega} \otimes \B(\ell^\infty S)$ gives a normal Schur multiplier $H_{m(\omega)}: \B(\ell^2 S) \to \B(\ell^2 S)$ of symbol $m(\omega)$ and norm bounded by that of $H_m$. Furthermore, the fact that the norm of $\ell^\infty(\Omega) \weaktensor \B(\ell^2 S) = \ell^\infty(\Omega;\B(\ell^2 S))$, as well as all its matrix amplifications, can be taken as a supremum in $\Omega$ easily yields that
  \[
    \big\| H_m: \B(\ell^2 S) \to \ell^\infty(\Omega) \weaktensor \B(\ell^2 S) \big\|_\cb
    \, = \, 
    \max_{\omega \in \Omega} \Big\{ \big\| H_{m(\omega)}: \B(\ell^2 S) \to \B(\ell^2 S) \big\|_\cb \Big\}.
  \]
  For each $\omega \in \Omega$, there is a Hilbert space $H_\omega$ such that 
  \[
    m_{s,t} = \big\langle \Xi_s^{[\omega]}, \Sigma_t^{[\omega]} \big\rangle
  \]
  and satisfying that the complete norm of $H_{m(\omega)}$ is equal to $\| \Xi^{[\omega]} \| \, \| \Sigma^{[\omega]} \|$. You can construct $H$ as the direct sum of all $H_\omega$, ie $H = \ell^2(\Omega;\{H_\omega\}_{\omega \in \Omega})$, and the maps $\Xi: S \times \Omega \to H$ and $\Sigma: S \times \Omega \to H$ by $\Xi_s(\omega) = \Xi_s^{[\omega]} \in H_\omega \subset H$ and a similar formula for $\Sigma$. A trivial calculation yields the result. 
  
  The case in which $\Omega$ is diffuse, although not conceptually different, runs into technical difficulties due to measurability. Let us start noticing that, if $(m^\alpha)_\alpha$ is a bounded sequence of symbols and $H_{m^\alpha} \to H_{m}$ in the weak-$\ast$ topology, then $\lim_\alpha m_{s,t}^\alpha(\omega) = m_{s,t}(\omega)$ $\mu$-almost everywhere. Let $\Se$ be the $\sigma$-algebra of $(\Omega, \mu)$. Take a filtration of $\sigma$-algebras $\Se_\alpha \subset \Se_{\alpha + 1} \dots \subset \Se$ such that each of the sigma algebras $\Se_\alpha$ is given by the power set of a partition
  \[
    \Omega = \bigcup_{j \geq 0} \Omega_j^{\alpha}
  \] 
  into sets of finite measure. We will also assume that the union of all the partitions $\Se_\alpha$ generated $\Se$. As a consequence, we have that the subalgebras $L^\infty(\Omega; \Se_\alpha) \subset L^\infty(\Omega)$, are thus isomorphic to $\ell^\infty$ and the union of them is a weak-$\ast$ dense subalgebra of $L^\infty(\Omega)$. Let us denote by $\EE_\alpha: L^\infty(\Omega) \to L^\infty(\Omega; \Se_\alpha) \cong \ell^\infty$ the conditional expectation associated to the sub-$\sigma$-algebra $\Se_\alpha$. We will denote as well by $\EE_\alpha$ the tensor amplification $\EE_\alpha \otimes \Id: L^\infty(\Omega) \weaktensor \B(\ell^2 S) \to \ell^\infty \weaktensor \B(\ell^2 S)$. Clearly we have that $\EE_\alpha \circ H_{m}$ converge in the weak-$\ast$ topology to $H_m$. But it also holds that $\EE_\alpha \circ H_m  = H_{\EE_\alpha(m)}$ and therefore $\EE_\alpha(m_{s,t})$ converges almost everywhere to $m_{s,t}$
  \begin{center}~
    \xymatrix@C=1.5cm@R=1cm{
      \B(\ell^2 S) \ar[r]^-{H_m} & L^\infty(\Omega) \weaktensor \B(\ell^2 S) \ar[d]^{\EE_\alpha} \\
      \B(\ell^2 S) \ar[u]^{=} \ar[r]^-{H_{\EE_\alpha(m)}} & \ell^\infty \weaktensor \B(\ell^2 S).
    }
  \end{center}
  But now, by applying the discrete case, we have that each of the symbols $m^\alpha_{s,t} = \EE_\alpha(m_{s,t})$ has a decomposition as $\langle \Xi_s^{\alpha}(\omega), \Sigma_t^{\alpha}(\omega) \rangle$, where $\Xi$, $\Sigma$ are $\Se_\alpha$-measurable functions in $L^\infty(\Omega \times S; H_\alpha)$, for some Hilbert space $H_\alpha$ and they reach the norm of $H_m^\alpha$. We can take the ultraproduct Hilbert space $H = \prod_{\alpha, \U} H_\alpha$, for some proper ultrafilter $\U$, and notice that we can construct $\Xi$ and $\Sigma$ by identifying $L^\infty(\Omega \times S;H)$ with $\prod_{\alpha, \U} L^\infty(\Omega \times S;H_\alpha)$ so that $\Xi = (\Xi^\alpha)_\alpha^\U$ and $\Sigma = (\Sigma^\alpha)_\alpha^\U$. We have that
  \begin{eqnarray*}
    m_{s,t}(\omega)
      & = & \big\langle X_s(\omega), \Sigma_t(\omega) \big\rangle\\
      & = & \lim_{\alpha \to \U} \big\langle X_s^\alpha(\omega), \Sigma_t^\alpha(\omega) \big\rangle \\
      & = & \lim_{\alpha \to \U} m_{s,t}^\alpha(\omega) \, = \, m_{s,t}(\omega)
  \end{eqnarray*}
  where the last equality follows almost everywhere. The identity for the norms follows similarly. 
\end{proof}

\begin{remark} \normalfont
  \
  \begin{enumerate}[leftmargin=1.35cm, label={\rm \textbf{(R.\arabic*)}}, ref={\rm \textbf{(R.\arabic*)}}]
    \item Like in the case of Schur multipliers the complete boundedness of the $L^\infty(\Omega)$-valued multipliers is automatic, see \cite{BozejkoFendler1984HSFourier}. This can be obtained from the automatic complete boundedness results of \cite{Smith1991} after noting that the maps above are $L^\infty(\Omega) \weaktensor \ell^\infty(S)$-bimodular. Like in the case of Fourier multipliers, the decomposition above extends to equivariant operators $\Phi: \L \Gamma \to L^\infty(\Omega) \rtimes_\theta \Gamma$ only when the operator $\Phi$ us completely bounded, see the proof of Theorem \ref{thm:EquivariantMap} \ref{itm:EquivariantMap3}. 
  
    \item The proof above works verbatim to prove that the alternative $L^\infty(\Omega)$-valued Herz-Schur multiplier given by
    \[
      H_m \Big( \sum_{s,t \in S} f_{s,t}(\omega) \otimes e_{s,t} \Big)
      \, = \, 
      \sum_{s,t \in S} m_{s,t}(\omega) \, f_{s,t}(\omega) \otimes e_{s,t},
    \] 
    is bounded/completely bounded if and only if there exists a Hilbert space $H$ and $\Xi, \Sigma \in L^\infty(\Omega \times S;H)$ such that 
    \[
      m_{s,t}(\omega) = \big\langle \Xi_s(\omega), \Sigma_t(\omega) \big\rangle
    \]
    for almost every $\omega \in \Omega$. In that case, its norm is again given by
    \[
      \begin{split}
        \| H_m: L^\infty(\Omega) \weaktensor & \B(\ell^2) \to L^\infty(\Omega) \weaktensor \B(\ell^2) \| \\
          & = \inf \Big\{ \| \Xi \|_{L^\infty(\Omega;H)} \, \| \Sigma \|_{L^\infty(\Omega;H)} \Big\},
      \end{split}
    \]
    where the infimum is taken over all such decompositions.
        
    \item 
    There is an alternative route to prove the point above that, although much less direct, would give a far reaching result. 
    Such alternative route would be obtained from an intricate combinations of scattered results in the literature as follows. Given an hyperfinite von Neumann algebra $\M$, the results in \cite{ChatterjeeSinclair1992, ChatterjeeSmith1993} imply that the map from the, so called, central Haagerup tensor product $\M \otimes_{\Zent, h} \M$ into $\CB^\sigma(\M, \M)$ given by extension of $x \otimes y \mapsto \ell(x) r(y)$, where $\ell, r$ represent the left and right multiplication operators on $\M$ respectively, is isometric. Its image is point weak-$\ast$ dense inside the normal and decomposable maps of $\M$, which by \cite{Haagerup1985Dec} are all completely bounded normal maps. Extending the isometry from $\M \otimes_{\Zent, h} \M$ to its extended tensor product version, in the sense of \cite{EffRu2003, BleSmi1992} will make the map surjective. Given a sub-von Neumann algebra $\N \subset \M$, we would have to follows the steps of \cite{Smith1991} to construct an isomorphism between the normal $\N'$-$\N'$-bimodular maps $\CB^\sigma_{\N', \N'}(\M, \M)$ and a central version of the extended Haagerup tensor $\N \otimes_{\Zent, eh} \N$. Taking $\M = L^\infty(\Omega) \weaktensor \B(\ell^2 S)$ and $\N = L^\infty(\Omega) \weaktensor \ell^\infty(S)$ will give the result above.  
  \end{enumerate}
\end{remark}

We can proceed to prove the main Theorem of the section.

\begin{proof}[Proof (of Theorem \ref{thm:EquivariantMap})]
  For \ref{itm:EquivariantMap1}, the only if part is trivial. Indeed, it is trivial to see that if the map $\lambda_g \mapsto \varphi_g \rtimes \lambda_g$ is multiplicative, then $\varphi: \Gamma \to L^\infty(\Omega; \TT)$ is a multiplicative $1$-cocycle. To see that every multiplicative $1$-cocycle induces a normal $\ast$-homomorphism, we will use the following version of the Fell absorption principle. Let $W: L^2(\Omega) \otimes_2 \ell^2(\Gamma) \to L^2(\Omega) \otimes_2 \ell^2 \Gamma$, be the unitary given by extension of 
  \[
    \xi \otimes \delta_h \xmapsto{\quad W \quad} \theta_h^{-1}(\overline{\varphi}_h)\xi \otimes \delta_h.
  \]
  The following diagram commutes
  \begin{equation*}
    \label{dia:FellAbsorption}
    \xymatrix@C=2cm{
      L^2(\Omega) \otimes_2 \ell^2(\Gamma) \ar[r]^{W} \ar[d]^{\varphi_k \rtimes \lambda_k}
        & L^2(\Omega) \otimes_2 \ell^2(\Gamma) \ar[d]^{\1 \otimes \lambda_k} \\
      L^2(\Omega) \otimes_2 \ell^2(\Gamma) \ar[r]^{W} & L^2(\Omega) \otimes_2 \ell^2(\Gamma).
    }
  \end{equation*}
  Indeed, we have that
  \begin{eqnarray*}
    \big[ (\1 \otimes \lambda_k) \, W \big] (\xi \otimes \delta_h) 
      & = & (\1 \otimes \lambda_k) \, (\theta_{h^{-1}}(\overline{\varphi}_h) \xi \otimes \delta_h\\
      & = & \theta_{h^{-1}}(\overline{\varphi}_h) \xi \otimes \delta_{k \, h}
  \end{eqnarray*}
  while
  \begin{eqnarray*}
    \big[ W \, (\varphi_k \rtimes \lambda_k) \big](\xi \otimes \delta_h)
      & = & W \big[ \theta_{k h}^{-1}(\varphi_k) \xi \otimes \delta_{k h} \big]\\
      & = & \theta_{k \, h}^{-1}(\overline{\varphi}_{k \, h}) \theta_{k \, h}^{-1}(\varphi_k) \xi \otimes \delta_{k h} \quad = \quad \theta_{h}^{-1}(\overline{\varphi}_{h}) \xi \otimes \delta_{k h}
  \end{eqnarray*}
  Therefore $W$ spatially implements the homomorphism $\lambda_g \mapsto \varphi_g \rtimes \lambda_g$. 

  For \ref{itm:EquivariantMap2} the if side is very similar. Indeed,
  the map $\lambda_g \mapsto \kappa_g \rtimes \lambda_g$ extends to a normal $\ast$-homomorphism $\pi: \L \Gamma \to L^\infty(\Omega;\B(H)) \rtimes_{\theta \otimes \Id} \Gamma$. To see that, just notice that, again by a Fell like absorption principle, there is a unitary $W: H \otimes_2 L^2(\Omega) \otimes_2 \ell^2 \Gamma \to H \otimes_2 L^2(\Omega) \otimes_2 \ell^2 \Gamma$ such that the following diagram commutes
  \begin{equation}
    \xymatrix@C=2cm{
      H \otimes_2 L^2(\Omega) \otimes_2 \ell^2 \Gamma \ar[r]^{W} \ar[d]^{(\kappa_g \otimes \lambda_g)}
        & H \otimes_2 L^2(\Omega) \otimes_2 \ell^2 \Gamma \ar[d]^{(\1 \otimes \lambda_g)}\\
      H \otimes_2 L^2(\Omega) \otimes_2 \ell^2 \Gamma \ar[r]^{W}
        & H \otimes_2 L^2(\Omega) \otimes_2 \ell^2 \Gamma ,
    }
  \end{equation} 
  where, after identifying $H \otimes_2 L^2(\Omega)$ with $L^2(\Omega;H)$, $W$ is given by
  \[
    W ( \xi(\omega) \otimes \delta_h )
    \, = \,
    \kappa_{h^{-1}}(\omega) \, \xi(\omega) \otimes \delta_h.
  \]
  The map $\Phi$  of the form $\Phi(x) = V^\ast \pi(x) \, V$, where $V: L^2(\Omega) \otimes_2 \ell^2(\Omega) \to H \otimes_2 L^2(\Omega) \otimes_2 \ell^2(\Gamma)$ is given by $\eta \delta_k \mapsto \xi \otimes \eta \otimes \delta_k$, is completely positive by Stinespring's Theorem. 
  
  To see that any completely positive equivariant map is of this form we will use the theory of $W^\ast$-Hilbert modules \cite{Maunilov2005Hilbert, Lance1995}. Define the $L^\infty(\Omega)$-valued inner product over $\X_0 = L^\infty(\Omega) \algtensor \CC[\Gamma]$ by 
  \[
    \Big\langle \sum_{g \in \Gamma} F_g \otimes \delta_g, \sum_{k \in \Gamma} G_k \otimes \delta_k \Big\rangle_\Phi
    = \sum_{g \in \Gamma} \sum_{k \in G} \theta_{g^{-1}} \big( \overline{F}_g \, G_k \big) \, \varphi_{g^{-1} \, k}.
  \]
  Observe that $\langle \cdot, \cdot \rangle_\Phi$ is positive definite by Lemma \ref{lem:CP}. Let us denote by $\X$ the completion of $\X_0$ modulo its nulspace for the seminorms
  \[
    x \longmapsto \phi \big( \langle x, x \rangle^\frac12  \big),
  \]
  for every $\phi \in L^1(\Omega)$. Using \cite[Theorem 2.5.]{JungeSherman2005}, we have that $\X$ embeds as a complemented $L^\infty(\Omega)$-module of $L^\infty(\Omega;H)$ for some Hilbert space $H$. Observe also that every element $x_0 \in \CC \, \1_{\Omega} \algtensor \CC[\Gamma]$ extends to a constant vector $x = \1_\Omega \otimes \xi \in L^\infty(\Omega;H)$. For every $h \in \Gamma$, let $L_h$ be the operator given linear extension of $L_h(F \otimes \delta_k) = F \otimes \delta_{h \, k}$. It holds that
  \begin{equation}
    \label{eq:InvariantOperator}
    \big\langle L_h(x), L_h(y) \big\rangle_\Phi
    \, = \,
    \theta_{h} \langle x, y\rangle_\Phi.
  \end{equation}
  Observe that equation \eqref{eq:InvariantOperator} implies that $L_h$ extends to an isometric operator $\pi_h$ over $\X$ and to a unitary operator $\pi_h$ acting on the Hilbert space $L^2(\Omega;H)$ whose inner product is given by $\mu \circ \langle \cdot, \cdot \rangle_\Phi$. Equation \eqref{eq:InvariantOperator} also yields that, for every $\xi \in L^\infty(\Omega; H)$, $[\pi_h \xi](\omega) = \kappa_h(\omega) \xi(\theta_h^{-1}\omega)$, for some unitary $\kappa_h(\omega) \in \U(H)$. A straightforward computation gives that $\kappa$ is a unitary $1$-cocycle and that
  \[
    \varphi_g(\omega) = \langle \xi, \kappa_g(\omega) \, \xi \rangle,
  \]
  where $\xi \in H$ is the vector associated to the constant extension of $x_0 = \1 \otimes \delta_e$.
  
  For \ref{itm:EquivariantMap3} first we check that every $\varphi$ satisfying that $\varphi_{g^{-1} h} = \big\langle \theta_g^{-1} \Xi_g, \theta_g^{-1} \Sigma_h \big\rangle$ is completely bounded. For that, use that the embedding $\iota: L^\infty(\Omega) \rtimes_\theta G \into L^\infty(\Omega) \weaktensor \B(\ell^2 S)$ is given by
  \begin{equation}
    \iota \Big( \sum_{g \in \Gamma} f_g \rtimes \lambda_g \Big)
    \, = \,
    \sum_{g, k \in \Gamma} \theta_{g^{-1}}(f_{g \, h^{-1}}) \otimes e_{g, h}.
  \end{equation}
  We have that the following diagram commutes 
  \begin{equation}
    \xymatrix@C=1.25cm@R=1cm{
     \L \Gamma \ar[r]^{\iota} \ar[d]^{\Phi} 
        & \B(\ell^2 \Gamma) \ar[d]^{H_m}\\
      L^\infty(\Omega) \rtimes_\theta \Gamma \ar[r]^{\iota} 
        & L^\infty(\Omega) \weaktensor \B(\ell^2 \Gamma),
    }
  \end{equation}
  where $H_m$ is the $L^\infty(\Omega)$-valued Schur multiplier of symbol $m_{g, h}(\omega) = \theta_{g^{-1}}(\varphi_{g \, h^{-1}}) = \langle \Xi_g(\omega), \Sigma_h(\omega) \rangle$, which is completely bounded. 
  For the reciprocal we define the injective and normal $\ast$-homomorphism
  \[
    L^\infty(\Omega) \weaktensor \B(\ell^2 \Gamma) \xrightarrow{\quad \pi \quad} \big( L^\infty(\Omega) \rtimes_\theta \Gamma \big) \weaktensor \B(\ell^2 \Gamma)
  \]
  given by 
  \[
    \pi \Big( \sum_{g,h \in \Gamma} f_{g,h} \otimes e_{g,h} \Big)
    \, = \,
    \sum_{g,h \in \Gamma} \big( \theta_{g^{-1}}(f_{g,h}) \rtimes \lambda_{s^{-1} t} \big) \otimes e_{s,t}.
  \]
  In an abuse of notation we will also denote by $\pi$ the $\ast$-homomorphism given $\pi: \B(\ell^2 \Gamma) \to \L \Gamma \weaktensor \B(\ell^2 \Gamma)$ obtained when $L^\infty(\Omega) = \CC \1$.  
  Those maps satisfy the following intertwining identity
  \begin{center}~
    \xymatrix@C=1.5cm@R=1.25cm{
      \B(\ell^2 \Gamma) \ar[r]^-{\pi} \ar[d]^-{H_{\theta_g(\varphi_{g^{-1} h})}}
        & \L  \Gamma \weaktensor \B(\ell^2 \Gamma) \ar[d]^-{\Phi \otimes \Id}\\
      L^\infty(\Omega) \weaktensor \B(\ell^2 \Gamma) \ar[r]^-{\pi}
        & \big( L^\infty(\Omega) \rtimes_\theta \Gamma \big) \weaktensor \B(\ell^2 \Gamma).
    }
  \end{center}
  Observe that the symbol $\theta_g(\varphi_{g^{-1} h})$ induces a bounded multiplier if $\Phi \otimes \Id$ is bounded or, equivalently, if $\Phi$ is completely bounded. But, by Lemma \ref{lem:LinftyValuedSchurs}, if the $L^\infty(\Omega)$-valued Schur multiplier is bounded, it admits a decomposition
  \[
    \theta_g(\varphi_{g^{-1} h}) = \big\langle \Xi_{g}', \Sigma_h' \big\rangle
  \]
  and by setting $\Xi_g = \Xi_{g^{-1}}'$ and $\Sigma_g = \Sigma_{g^{-1}}'$ we obtain a decomposition like that of \eqref{eq:Decomposition}.
\end{proof}

\begin{remark} \normalfont
  \label{rmk:NecessityOfCB}
  In direct analogy with the case of Fourier multipliers, see \cite{BozejkoFendler1984HSFourier}, the complete boundedness of the equivariant map $\Phi:\L \Gamma \to L^\infty(\Omega) \rtimes_\theta \Gamma$ given by extension of $\lambda_g \mapsto \varphi_g \rtimes \lambda_g$ is only required to prove that the factorization \eqref{eq:Decomposition} is necessary. Quantitatively, we have that if $m_{g,h} = \theta_{g^{-1}}(\varphi_{g \, h^{-1}})$ and $H_m$ is its associated $L^\infty(\Omega)$-valued Schur multiplier, we have that 
  \begin{eqnarray}
    & & \big\| \Phi: \L \Gamma \to L^\infty(\Omega) \rtimes_\theta \Gamma \big\|
        \, \leq \, \big\| H_m: L^\infty(\Omega) \weaktensor \B(\ell^2 \Gamma) \to L^\infty(\Omega) \weaktensor \B(\ell^2 \Gamma) \big\| \label{eq:SchurToFourier}\\
    & & \big\| H_m: L^\infty(\Omega) \weaktensor \B(\ell^2 \Gamma) \to L^\infty(\Omega) \weaktensor \B(\ell^2 \Gamma) \big\| 
        \, \leq \, \big\| \Phi: \L \Gamma \to L^\infty(\Omega) \rtimes_\theta \Gamma \big\|_\cb, \nonumber
  \end{eqnarray}
  where equation \eqref{eq:SchurToFourier} works also if you use completely bounded norms in both sides of the inequality. The factorization \eqref{eq:Decomposition} follows by applying Lemma \ref{lem:LinftyValuedSchurs} to $m_{g,h} = \theta_{g^{-1}}(\varphi_{g \, h^{-1}})$.
\end{remark}

\section{Transference for actions with an invariant mean}

Our proof leans in a key way on the following technical lemma, which follows techniques of almost multiplicative maps developed in \cite[Theorem B/Corollary 2.3]{CasParPerrRic2014}, which generalize both pseudolocalization techniques for Fourier multipliers and the noncommutative Powers-St{\o}rmer's inequlity. Some consequences of those result are also obtained in \cite{Ricard2015Mazur}. We recall that the technniques use here are similar to those in \cite{Gon2018CP}. 

\begin{theorem}[{\cite[Theorem B/Corollaries 2.3/2.4.]{CasParPerrRic2014}}]
  \label{thm:almostMultiplicativemap}
  Let $(\M, \tau)$ be a semifinite von Neumann algebra with a normal, faithful and semifinite tracial weight and let $R: \M \to \M$ be a positive subunital map such that $\tau \circ T \leq \tau$ map, we have that
  \begin{enumerate}[leftmargin=1.25cm, label={\rm \textbf{(\roman*)}}, ref={\rm (\roman*)}]
    \item For every $\xi \in L^2(\M)_+$
    \[    
      \big\| R(\xi^\theta) - \xi^\theta \big\|_{\frac{2}{\theta}}
      \, \leq \,
      C \, \big\|  R(\xi) - \xi \big\|_{2}^{\frac{\theta}{2}} \, \| \xi \|_{2}^{\frac{\theta}{2}}.
    \]
    \item For every $\xi \in L^2(\M)$
    \[
      \big\| R(u \, |\xi|^\theta) - u \, |\xi|^\theta \big\|_{\frac{2}{\theta}}
      \, \leq \,
      C \, \big\|  R(\xi) - \xi \big\|_{2}^{\frac{\theta}{2}} \, \| \xi \|_{2}^{\frac{\theta}{2}},
    \]
    where $\xi = u \, |\xi|$ is the polar decomposition of $\xi$.
  \end{enumerate}
\end{theorem}

We also cite the following lemma, although it can be obtained from Theorem \ref{thm:almostMultiplicativemap} above by taking $R(\xi) = u \, \xi \, u^\ast$, for every unitary in the spectral calculus of $x$ and using ultraproduct techniques.
 
\begin{lemma}[{\cite[Lemma 2.4/ Lemma 2.6]{Ricard2015Mazur}}]
  \label{lem:Approximate}
  Let $\M$ be a von Neumann algebra and assume $\M$ is semifinite with a n.s.f. trace $\tau: \M_+ \to [0,\infty]$. Let $\xi \in L^2(\M)$ be a unit vector and $\xi = u \, |\xi|$ be its polar decomposition. Fix $1 \leq p \leq \infty$ and notice that $|\xi|^\frac{2}{p} \in L^p(\M)$. For every $x \in \M$ and $\varepsilon > 0$, there is a $\delta > 0$ such that
  \begin{enumerate}[leftmargin=1.25cm, label={\rm \textbf{(\roman*)}}, ref={\rm (\roman*)}]
    \item $\displaystyle{\big\| [x, \xi] \big\|_{2} < \delta} \quad \implies  \quad \displaystyle{\big\| \big[x, |\xi|^\frac{2}{p} \big] \big\|_{p} < \varepsilon}.$
    \item $\displaystyle{\big\| \big[ x, \xi \big] \big\|_{2} < \delta} \quad \implies  \quad \displaystyle{\big\| \big[x, u\, |\xi|^\frac{2}{p} \big] \big\|_{p} < \varepsilon}.$
  \end{enumerate}
\end{lemma}


With that lemma at hand we can prove the following theorem.

\begin{proof}[Proof (of Theorem \ref{thm:EquivariantMap})]
  Let $\iota: \L \Gamma \into L^\infty(\Omega) \rtimes_\theta \Gamma$ be the natural embedding that sends $\lambda_g$ to $\1 \rtimes \lambda_g$. Observe that if $\mu(\Omega) = 1$ we are already finishes since $\iota$ is trace preserving, ie $\tau_\rtimes \circ \iota = \tau$, and thus $\iota$ extends to all $L^p$-space. In the case of $\mu(\Omega) = \infty$ we have that, by the fact that $L^\infty(\Omega)$ admits a left-invariant mean $m:L^\infty(\Omega) \to \CC$, there is a unit net of vectors $\xi_\alpha \in L^2(\Omega)$ such that, for every $g \in \Gamma$,
  \[ 
    \big\| \theta_g(\xi_\alpha) - \xi_\alpha \big\|_2 \to 0.
  \]
  The vectors also satisfy that $m$ is any weak-$\ast$ accumulation point of the states $f \mapsto \langle \xi_\alpha, f x_\alpha\rangle$. 
  Let us define $J_p^\alpha: L^p(\L \Gamma) \to L^p(\Omega \rtimes_\theta \Gamma)$ by 
  \[
    J_p^\alpha(x) = u_\alpha \, |\xi_\alpha|^\frac{1}{p} \, \iota(x) \, |\xi_\alpha|^\frac{1}{p},
  \]
  where $x_\alpha = u \, |\xi_n|$ is the polar decomposition and we are identifying $\xi_\alpha \in L^2(\Omega)$ with $\xi_\alpha \rtimes \lambda_e \in L^2(\Omega \rtimes_\theta \Gamma)$. 
  We have that
  \begin{enumerate}[leftmargin=1.25cm, label={\rm \textbf{(\arabic*)}}, ref={\rm \textbf{(\arabic*)}}]
    \item \label{itm:IsometricInclusionProof.1}
    $\displaystyle{\big\| J_p^\alpha: L^p(\L \Gamma) \to L^p(\Omega \rtimes_\theta \Gamma) \big\| \leq 1}.$
    \item \label{itm:IsometricInclusionProof.2}
    $\displaystyle{\lim_{\alpha} \big\langle J_p^\alpha(x), J_p^\alpha(y) \big\rangle = \langle x, y \rangle}$ for every $x \in L^p(\L \Gamma)$ and $y \in L^q(\L \Gamma)$ with $\frac1{p} + \frac1{q} = 1$.
  \end{enumerate}
  The proofs of both points are straightforward generalizations of the results in \cite{NeuRic2011,CasSall2015,Gon2018CP}. Indeed, \ref{itm:IsometricInclusionProof.1} is obvious for
  $p = \infty$. By interpolation it suffices to be proven in the case of $p=1$. Let $x \in L^1(\L \Gamma)$ an element of unit norm. We can take $y$, $z \in \CC[\Gamma]$ such that $\| z - y , z \|_1 \leq \epsilon$, for every $\epsilon > 0$. To see that just notice that there are elements $y'$, $z' \in L^2(\L \Gamma)$ of unit norm and such that $x = y \, z$. By density, there are elements $y$ and $z$ in $\CC[\Gamma]$ such that $\| z - z' \|_2 \leq \delta$ and $\| y - y' \|_2 \leq \delta$. A calculation yields that
  \[
    \| x - y \, z \|_2  \leq \| y \, (z - z')\|_1 + \| (y - y') \, z' \|_1 \leq 2 \delta + \delta^2,
  \]
  and so, taking $\delta$ small enough gives the desired identity. Now, we have that
  \begin{eqnarray*}
    \big\| J_p^\alpha(x) \big\|_1
      & =    & \big\| J_p^\alpha( y \, z) \big\|_1 + \big\| J_p^\alpha( x - y \, z ) \big\|_1\\
      & =    & \big\| u_\alpha \, |\xi_\alpha| \, \iota(y \, z) \, |\xi_\alpha| \big\|_1 
               + \big\| u_\alpha \, |\xi_\alpha| \, \iota( x - y \, z ) \, |\xi_\alpha| \big\|_1
               \, = \, \mathrm{I} + \mathrm{II}
  \end{eqnarray*}
  The term $\mathrm{II}$ can be bounded by
  \[
    \big\| u_\alpha \, |\xi_\alpha| \, \iota( x - y \, z ) \, |\xi_\alpha| \big\|_1
    \leq
    \big\| u_\alpha \, |\xi_\alpha| \big|_2 \, \big\| \iota( x - y \, z ) \big\|_\infty \, \big\| |\xi_\alpha| \big\|_2
    \leq 
    \big\|  x - y \, z \big\|_1
    \leq \epsilon.
  \]
  While for $\mathrm{II}$, we have that
  \[
    \mathrm{II}
    \leq
    \big\| u_\alpha \, |\xi_\alpha| \, \iota(y) \big\|_2 \, \big\| \, \iota(z) \, |\xi_\alpha| \big\|_2
    = 
    A \, B.
  \]
  Both terms, $A$ and $B$ are estimated similarly. Let $y$ be the finite sum $\sum_{g \in \Gamma} y_g \lambda_g$. We have that $A$ is bounded by
  \begin{eqnarray*}
    A^2 \, = \, 
    \big\| u_\alpha \, |\xi_\alpha| \, \iota(y) \big\|_2^2
      & = & \big\| |\xi_\alpha| \, \iota(y) \big\|_2^2\\
      & = & \tau_\rtimes \bigg\{ 
            \Big( \sum_{g \in \Gamma} \overline{y}_g \, \1_\Omega \rtimes \lambda_{g^{-1}} \Big) 
            \, |\xi_\alpha|^2 \,
            \Big( \sum_{g \in \Gamma} y_h \, \1_\Omega \rtimes \lambda_{h} \Big)
            \bigg\}\\
      & = & \sum_{g \in \Gamma} \overline{y}_g \, y_g \, \int_{\Omega} \theta_{g^{-1}}(|\xi_\alpha|^2) \, d \mu
            \, = \, \| y \|_2^2 \, \leq \, (1 + \epsilon)^2.
  \end{eqnarray*}
  A similar estimate holds for $B$. Joining both and using that $\epsilon$ is arbitrary gives \ref{itm:IsometricInclusionProof.1}. 
  For \ref{itm:IsometricInclusionProof.2}, we start choosing $x,y \in \CC[\Gamma]$, given by $x = \sum_{g \in \Gamma} x_g \, \lambda_g$ 
  and $y = \sum_{h \in \Gamma} y_h \, \lambda_h$ and notice that 
  \begin{eqnarray}
    \lim_{\alpha} \big\langle J_p^\alpha(x), J_q^\alpha(x) \big\rangle
      & = & \lim_{\alpha} \tau_\rtimes 
            \big\{ 
              |\xi_\alpha|^\frac1{p} \, \iota(x^\ast) \, |\xi_\alpha|^\frac1{p} \, u_\alpha^\ast 
              u_\alpha \, |\xi_\alpha|^\frac1{q} \, \iota(y) \, |\xi|^\frac1{q}
            \big\} \label{eq:central}\\
      & = & \lim_{\alpha} \tau_\rtimes \big\{ \xi_\alpha^\ast \, \iota(x^\ast \, y) \, \xi \big\} \nonumber\\
      & = & \lim_{\alpha} \tau_\rtimes \Big\{ \sum_{g, h \in \Gamma} \overline{x}_g \, y_{h} \, \xi_\alpha \, \theta_{g^{-1} h}(\xi_\alpha) \rtimes \lambda_{g^{-1} h} \Big\} \nonumber\\
      & = & \lim_{\alpha} \sum_{g \in \Gamma} \overline{x}_g \, y_{g} \int_{\Omega} |\xi_\alpha(\omega)|^2 \, d \mu(\omega) \quad = \quad \langle x, y \rangle, \nonumber
  \end{eqnarray}
  where we have used Lemma \ref{lem:Approximate} in \eqref{eq:central}. The identity above can be extended to elements not in $\CC[\Gamma]$. Indeed, fix $y \in \CC[\Gamma]$ and choose $x \in L^p(\L \Gamma)$, where $p < \infty$. Approximating $x$ in the $L^p$-norm by elements in $\CC[\Gamma]$ gives that the inequality duality bracket extends to $L^p(\L \Gamma) \times \CC[\Gamma]$, but now, approximating $y$ by elements in $\CC[\Gamma]$ in the $L^q$-norm, or in the weak-$\ast$ topology if $q = \infty$ gives that the bracket above is well defined for pairs in $L^p(\L \Gamma) \times L^q(\L \Gamma)$. Now, \ref{itm:IsometricInclusionProof.1} and \ref{itm:IsometricInclusionProof.2} imply that the ultraproduct $J_p = J_p^\U$, given by
  \[
    J_p(x) = \big( J_p^\alpha(x) \big)_\alpha + n_{\U} \in \prod_{\U} L^p(\Omega \rtimes_\theta \Gamma),
  \]
  is an isometry. Indeed, that it is a contraction follows from \ref{itm:IsometricInclusionProof.1}. For the isometry, we notice that the dual of $L^p(\Omega \rtimes_\theta \Gamma)^\U$ contains $L^q(\Omega \rtimes \Gamma)^\U$ isometrically. Therefore
  \begin{eqnarray*}
    \big\| j_p(x) \big\|_p
      & = & \sup_{\phi \in \Ball[(L^p(\Omega \rtimes \Gamma)^U)^\ast]} \big\langle x, y \big\rangle\\
      & \geq & \sup_{y \in \Ball[L^q]} \big\langle j_p(x), j_p(y) \big\rangle \quad = \quad \sup_{y \in \Ball[L^q]} \big\langle x, y \big\rangle \quad = \quad \| x \|_p
  \end{eqnarray*}
  The fact that $J_p$ intertwines $T_m$ and $\Id \rtimes T_m$ and the $\L \Gamma$-modularity follow immediately. 
\end{proof}

\begin{remark} \normalfont
  \label{rmk:Several}
  Several remarks are in order.
  \begin{enumerate}[leftmargin=1.35cm, label={\rm \textbf{(R.\arabic*)}}, ref={\rm \textbf{(R.\arabic*)}}]
     \item 
     First, in the definition of $J_p^\alpha$ the choice of putting half a power of $|\xi_\alpha|$ on the right and half on the left is arbitrary. We could have chosen instead any map of the form 
     \[
       J_p^\alpha(x) = u_\alpha \, |x_\alpha|^\frac{2 \, \theta}{p} \iota(x) |x_\alpha|^\frac{2 \, (1- \theta)}{p},
     \]
     where $0 \leq \theta \leq 1$, and the proof would work similarly. Indeed, all such map become independent of $\theta$ after taking the ultrapower by the asymptotic centrality of the vectors $|\xi_\alpha|^\beta \in L^\frac{2}{\beta}(\M)$ for any $\beta \in (0,2]$. The same follows for the choice of putting $u_\alpha$ at the left or at the right.
     \item 
     The statement of Theorem \ref{thm:IsometricInclusion} covers only the case of crossed products of an action on a measure space. But a similar result holds in the case of of trace preserving actions $\theta: \Gamma \to \Aut(\M, \varphi)$, where $\M$ is a semifinite von Neumann algebra and $\varphi$ is a n.s.f trace. In that setting the following two conditions are equivalent 
     \begin{itemize}
       \item There is a (non necessarily normal) state $\varpi: \M \to \CC$ that is $\theta$-invariant.
       \item There is a net of unit vectors $\xi_\alpha \in L^2(\M)$ such that $\big\| \theta_g(\xi_\alpha)  -\xi_\alpha \big\|_2 \to 0$.
     \end{itemize}
     Imposing that condition on the action gives that the maps $J_p^\alpha(x) = u_\alpha \, |\xi_\alpha|^\frac1{p} \, \iota(x) \, |\xi_\alpha|^\frac1{p}$, yield, after passing to an ultrapower, a complete isometry $J_p: L^p(\L \Gamma) \to L^p(\M \times_\theta \Gamma)^\U$. That isometry is also $\L \Gamma$-bimodular and intertwines the operators $T_m$ and $\Id \rtimes T_m$.
  \end{enumerate}
\end{remark}

In \cite{Gon2018CP} it was proved that, when the action $\theta: \Gamma \to \Aut(\Omega,\mu)$ is $\mu$-preserving and amenable in the sense of Zimmer, see \cite[Section 4.3]{BroO2008}, \cite{Zimmer1984} or the original sources of \cite{Zimmer1977, Zimmer1978Pairs, Zimmer1978, AdamEllGior1994}, then the natural embedding $L^\infty(\Omega) \rtimes_\theta \Gamma \into L^\infty(\Omega) \weaktensor \B(\ell^2 \Gamma)$ gives rise to a complete $L^p$-isometry
\[
  L^p( \Omega \rtimes_\theta \Gamma ) \xrightarrow{ \quad J_p \quad } \prod_\U L^p\big(\Omega; S^p[\ell^2 \Gamma]\big) 
\]
that intertwines the maps $\Id \rtimes T_m$ and $\Id \otimes H_m$, where $H_m$ is the Herz-Schur multiplier associated to $m \in \ell^\infty(\Gamma)$. Therefore, we have that
\begin{equation}
  \label{eq:FordwardTransference}
  \begin{split}
    \big\| \Id \rtimes & T_m: L^p(\Omega \rtimes_\theta \Gamma) \to L^p(\Omega \rtimes_\theta \Gamma) \big\|_\cb\\
      & \leq \big\| \Id \otimes H_m: L^p(\Omega; S^p[\ell^2 \Gamma]) \to L^p(\Omega; S^p[\ell^2 \Gamma]) \big\|\\
      & \leq \big\| H_m: S^p[\ell^2 \Gamma] \to S^p[\ell^2 \Gamma] \big\|.
  \end{split}
\end{equation}
Observe that, although if a group $\Gamma$ is Zimmer-amenable and has a probability measure preserving action then $\Gamma$ itself is amenable, there are plenty of Zimmer-amenable actions that preserve semifinite measures. For instance, if $\Gamma$ is a countable exact group, then it admits an topological action $\theta: \Gamma \to \Aut(X)$ on a compact space $X$ that is Zimmer amenable, see \cite{Ozawa2000Exact}. Take a probability measure $\mu \in \PP(X)$ with total support and such that the action $\theta$ is nonsingualr, ie such that for every $g \in \Gamma$, $\mu$ and $\theta_g^\ast \mu$ are mutually absolutely continuous. Then, the Maharam extension of $\theta$, see \cite{Maharam1953Integrals}, is both Zimmer-amenable and measure preserving. Recall that the Maharam extension of $\theta$ is given by $\hat{\theta}:\Gamma \to \Aut(X \times \RR, \mu \otimes \nu)$, where $d \nu(t) = e^{t} d t$ and the action is defined as
\[
  (x, s) \mapsto (\theta_g \, x, s + \log \omega_g(x)),
\]
where $\omega_g: X \to \RR$ is the Radon-Nykodim derivative of $\theta_g^\ast \mu$ with respect to $\mu$. Therefore each exact discrete group admits a measure preserving Zimmer-amenable action.
If the lower bound \eqref{eq:lowerTransferenceBnd} holds for a particular Zimmer-amenable action $\theta: \Gamma \to \Aut(X)$ preserving a $\sigma$-finite measure, then we will have that 
\begin{eqnarray*}
  \big\| T_m: L^p(\L \Gamma) \to L^p(\L \Gamma) \big\| 
    & \leq & \big\| \Id \rtimes T_m: L^p(\Omega \rtimes_\theta \Gamma) \to L^p(\Omega \rtimes_\theta \Gamma) \big\|\\
    & \leq & \big\| H_m: S^p[\ell^2 \Gamma] \to S^p[\ell^2 \Gamma] \big\|
\end{eqnarray*}
and the same would follow from complete norms, thus solving the open problem. Sadly, the conditions of being Zimmer amenable and having a $\Gamma$-invariant mean $m:L^\infty(\Omega) \to \CC$ give amenability for $\Gamma$ when they hold simultaneously. In summary, the situation that we obtain is the following
\begin{itemize}
  \item The embedding $\L \Gamma \into L^\infty(\Omega) \rtimes_\theta \Gamma$ yields an $L^p$-isometry $L^p(\L \Gamma) \to L^p(\Omega \rtimes_\theta \Gamma)^\U$ into the ultrapower that intertwines $T_m$ and $(\Id \rtimes T_m)^\U$ and therefore implies \eqref{eq:lowerTransferenceBnd} when $\Omega$ has a $\Gamma$-invariant mean.
  \item The embedding $L^\infty(\Omega) \rtimes_\theta \Gamma \to L^\infty(\Omega) \weaktensor \B(\ell^2 \Gamma)$ yields an $L^p$-isometry $L^p(\Omega \rtimes_\theta \Gamma) \to L^p(\Omega; S^p)^\U$ into the ultrapower that intertwines $\Id \rtimes T_m$ and $(\Id \rtimes H_m)^\U$ and therefore implies \eqref{eq:FordwardTransference} when $\theta: \Gamma \to \Aut(\Omega)$ is Zimmer amenable.
\end{itemize}
The intersection of both conditions only contains actions of amenable groups. It is open whether the conditions above are necessary for transference. Therefore the following problems remain widely open.

\begin{problem} \normalfont
  \label{prb:TrnsferenceActions}
  Let $\theta: \Gamma \to \Aut(\Omega, \mu)$ be a measure-preserving action on a $\sigma$-finite measure space and let $A(\Gamma)$ be the Fourier algebra of $\Gamma$.
  \begin{enumerate}[leftmargin=1.25cm, ref={\bf (P.\arabic*)}, label={\bf (P.\arabic*)}]
    \item \label{itm:TrnsferenceActions.1}
    Which actions satisfy that
    \[
      \big\| (\Id \rtimes T_m): L^p(\Omega \rtimes_\theta \Gamma \to L^p(\Omega \rtimes_\theta \Gamma) \big\|
      \, \leq \, 
      \big\| T_m: L^p(\L \Gamma) \to L^p(\L \Gamma) \big\|,
    \]
    for every $m \in A(\Gamma)$?
    \item \label{itm:TrnsferenceActions.2}
    Which actions satisfy that
    \[
      \big\| T_m: L^p(\L \Gamma) \to L^p(\L \Gamma) \big\| 
      \, \leq \, 
      \big\| (\Id \rtimes T_m): L^p(\Omega \rtimes_\theta \Gamma \to L^p(\Omega \rtimes_\theta \Gamma) \big\|,      
    \]
    for every $m \in A(\Gamma)$?
  \end{enumerate}
\end{problem}

Problem \ref{itm:TrnsferenceActions.1} has a positive solution when $\theta$ is Zimmer amenable, while Problem \ref{itm:TrnsferenceActions.2} has a positive solution when $\Omega$ has a $\Gamma$-invariant mean. For actions beyond those cases neither examples nor counterexamples are known.

\textbf{A correction on \cite{Gon2018CP}.}
  In the first paragraph after \cite[Corollary 3.2.]{Gon2018CP} the equality of the norms of $T_m$ and $\Id \rtimes T_m$ is stated for Zimmer-amenable actions. This is a clearly mistaken statement that did not appear in the ArXiv preprint of that same paper and was added during the process of providing more context to the results therein. The question that that paragraph was intending to address is whether the following identity holds
  \[
    \big\| \Id \rtimes T_m: L^p(\N \rtimes_\theta \Gamma) \to L^p(\N \rtimes_\theta \Gamma) \big\|
    = 
    \big\| \Id \rtimes T_m: L^p(\N \rtimes_\theta \Gamma) \to L^p(\N \rtimes_\theta \Gamma) \big\|_\cb
  \]
  where $\theta: \Gamma \to \Aut(\N, \tau)$ is any $\tau$-preserving action on a semifinite von Neumann algebra $\N$, provided that the algebra $\N$ is large enough in some sense. For instance, notice that if $\N$ absorbs any finite matrix $M_k(\CC)$, ie $M_k(\CC) \otimes \N \cong \N$, and $\theta$ is unitarily equivalent to $\Id \otimes \theta$, then identity above is automatic. That sort of equivariant decomposition follow for actions on amenable groups on finite McDuff factors by the work of Ocneanu \cite{Ocneanu1985}, that generalized to earlier results of Connes \cite{Connes1975Outer,Connes1977Periodic}.
  
  Indeed, let us denote by $\R$ the hyperfinite $\mathrm{II}_1$-factor. An von Neumann algebra $\N$ is said to be McDuff iff $\R \weaktensor \N \cong \N$. Recall similarly that, given two actions $\theta_1$ and $\beta$ of $\Gamma$ on a von Neumann algebra $\N$, they are said to be \emph{outer conjugate} iff there exists a normal $\ast$-isomorphism $\phi: \N \to N$ and a function $u: G \to \U(\N)$ such that 
  \[
    u_g \, \theta_g(x) \, u_g = \big( \phi \circ \beta_g \circ \phi^{-1} \big)(x),
  \] 
  for every $x \in \N$ and $g \in \Gamma$. Observe that the equation above forces $u$ to be a multiplicative $1$-cocycle. The equation above can easily be reinterpreted as saying that the group homomorphisms $\bar{\theta}$, $\bar{\beta}: \Gamma \to \Aut(\N)/\Inn(\N)$ are conjugate by the automorphism $\phi$. Whenever two actions are outerly conjugate, we have that the map
  \begin{equation}
    \label{eq:EquivariantIso}
    \sum_{g \in \Gamma} x_g \rtimes \lambda_g \longmapsto \sum_{g \in \Gamma} u_g \, \phi(x_g) \rtimes \lambda_g 
  \end{equation}
  is an equivariant and normal $\ast$-isomorphism of $\M \rtimes_\theta \Gamma$, Furthermore, if $\tau$ is a trace on $\M$, the automorphism is $\tau_\rtimes = (\tau \circ \EE)$-preserving iff $\phi$ is. The reciprocal is also true and those maps are all equivariant $\ast$-automorphisms. 
  
We also recall the following definitions related to action on von Neumann algebras, the interested reader can look in \cite{Ocneanu1985} and references therein for more information. 
\begin{itemize}
  \item 
  The action $\theta: \Gamma \to \Aut(\N)$ is \emph{free} iff for every $g \in \Gamma \setminus \{e\}$, we have that $\theta_g$ doesn't have nontrivial fixed points. Observe that, if $\theta$ is free then $\bar{\theta}: \Gamma \to \Out(\N)$ is faithful, otherwise if $\theta_g(x)$ were inner, then there will be a unitary $u \in \U(\N)$ such that $\theta_g(x) = u \, x \, u^\ast$ and $\theta_g(u) = u$. 
  \item 
  Let $(x_n)_n \in \ell^\infty[\N]$ be a centralizing sequence, ie a sequence such that for every $\psi \in \N_\ast$
  \[
    \lim_{n} \big\| [x_n, \psi] \big\|_\ast = 0.
  \]
  The action $\theta$ is \emph{centrally trivial} iff for every centralizing sequence $(x_n)_n$ we have that $\theta_g(x_n) - x_n \to 0$ 
  in the weak-$\ast$ topology for every $g \in \Gamma \setminus \{e\}$. 
  \item The action $\theta$ is \emph{strongly centrally trivial} iff for every $\Gamma$-invariant central projection $p \in \Prj(\N)$
  we have that $\theta{|}_{p \M p}$ is centrally trivial.
\end{itemize}

Notice that the last two definitions above collapse if $\N$ is a factor. We also have that the notion of being centrally trivial can be interpreted as saying that the ultraproduct action $\theta_g$ acts trivially over the relative commutant $\N_\U \cap \N'$. We have the following theorem.
  
\begin{theorem}[{\cite[Section 1.2]{Ocneanu1985}}]
  \label{thm:McDuffouterconjugate}
  Let $\M$ be a McDuff von Neumann algebra, $\Gamma$ an amenable group and $\theta: \Gamma \to \Aut(\M)$ a centrally free action, then $\theta$ and $\theta \otimes \Id_\R: \Gamma \to \Aut(\M \weaktensor \R)$ are outer conjugate.
\end{theorem}

With the previous theorem at hand we can easily obtain the automatic complete boundedness of $\Id \rtimes T_m$ for those families of actions. Indeed, lest $\Psi$ be the normal $\ast$-isomorphism 
\[
  \N \rtimes_\theta \Gamma \xrightarrow{\quad \Psi \quad} \big( \N \weaktensor \R ) \rtimes_{\theta \otimes \Id} \Gamma
\]
obtained from the outer conjugation of $\theta$ and $\theta \otimes \Id_\R$ as in \eqref{eq:EquivariantIso}. Observe that, since the map $\Phi$ is equivariant, it follows that the following diagram commutes
\begin{equation*}
  \xymatrix@C=2cm{
    \N \rtimes_\theta \Gamma \ar[d]^{\Id \rtimes T_m} \ar[r]^-{\Psi} & \big( \N \weaktensor \R )  \rtimes_{\theta \otimes \Id} \Gamma \ar[d]^{\Id_\R \otimes (\Id \rtimes T_m)}\\
    \N \rtimes_\theta \Gamma \ar[r]^-{\Psi} & \big( \N \weaktensor \R) \rtimes_{\theta \otimes \Id} \Gamma.
  }
\end{equation*}
Therefore
\[
  \begin{split}
    \big\| \Id \rtimes & T_m: \N \rtimes_\theta \Gamma \to \N \rtimes_\theta \Gamma \big\|\\
                       & \, = \, \big\| \Id_\R \otimes (\Id \rtimes T_m): \big( \N \weaktensor \R ) \times_{\theta \otimes \Id} \Gamma 
                                        \to \big( \N \weaktensor \R ) \rtimes_{\theta \otimes \Id} \Gamma \big\|\\
                       & \, = \, \big\| \Id_\R \otimes (\Id \rtimes T_m): \R \weaktensor \big( \N \times_{\theta} \Gamma \big) \to \R \weaktensor \big(\N  \times_{\theta} \Gamma \big) \big\|\\
                       & \, = \, \big\| \Id \rtimes T_m: \N \rtimes_\theta \Gamma \to \N \rtimes_\theta \Gamma \big\|_\cb.
  \end{split}
\]
Observe that, when $\M$ is a McDuff factor and not just an algebra, the uniqueness of the trace implies that $\Psi$ is trace preserving and therefore the diagram above passes trivially to $L^p$-spaces. In the non-factor case we need to use the following lemma whose proof is a trivial application of the Radon-Nykodim theorem

\begin{lemma}
  \label{lem:ChangeTrace}
  Let $\tau_i$, $i \in \{1,2\}$ be two faithful, normal and $\theta$-invariant tracial states over $\N$, where $\theta: \Gamma \to \Aut(\N)$ is an action. For every $1 \leq p \leq \infty$ there is a complete isometry $j_p$ such that
  \begin{equation*}
    \xymatrix@C=2cm{
      L^p(\N \rtimes_\theta \Gamma;\tau_1 \rtimes \tau_\Gamma) \ar[d]^{\Id \rtimes T_m} \ar[r]^-{j_p}
        & L^p(\N \rtimes_{\theta} \Gamma;\tau_2 \rtimes \tau_\Gamma) \ar[d]^{\Id \rtimes T_m}\\
      L^p(\N \rtimes_\theta \Gamma;\tau_1 \rtimes \tau_\Gamma) \ar[r]^-{j_p} 
        & L^p( \N \rtimes_{\theta} \Gamma;\tau_2 \rtimes \tau_\Gamma).
    }
  \end{equation*} 
\end{lemma} 

\begin{proof}
  We just add an sketch for completion. First notice that since both $\tau_1$ and $\tau_2$ are faithful tracial states we have that $\tau_2(x) = \tau(\delta \, x)$ for some positive central element in $L^1(\Zent(\N))$ with unit norm. The fact that both are $\theta$-invariant implies that $\delta$ is fixed by the action of $\theta$. Therefore $\iota(\delta) = \delta \rtimes \1$ is also central and we have that $\tau_2 \rtimes \tau_\Gamma(x) = \tau_1 \rtimes \tau_\Gamma(\iota(\delta) \, x)$, for every $x \in \N \rtimes_\theta \Gamma$. Let us denote $\iota(\delta)$. Simply by $\delta$. Notice that it is invertible as an element of $L^1(\Zent(\N))$ by the faithfulness of $\tau_1$ and $\tau_2$. We have that the map $j_p(x) = \delta^{-\frac1{p}} x$ is an $L^p$-isometry, equivariant and that the diagram above commutes. 
\end{proof}

\begin{corollary}
  \label{cor:AutomaticCB}
  Let $\theta: \Gamma \to \Aut(\M, \tau)$ be like in Theorem \ref{thm:McDuffouterconjugate} and assume that it is trace preserving. For every $1 < p < \infty$, we have that 
    \[
    \big\| \Id \rtimes T_m: L^p(\N \rtimes_\theta G) \to L^p(\N \rtimes_\theta G) \big\|
    = 
    \big\| \Id \rtimes T_m: L^p(\N \rtimes_\theta G) \to L^p(\N \rtimes_\theta G) \big\|_\cb
  \]
\end{corollary}

\begin{proof}
  Since, by Theorem \ref{thm:McDuffouterconjugate}, $\theta$ and $\theta \otimes \Id_\R$ are outer conjugate there is a $\ast$-isomorphism $\phi: \N \to \N \weaktensor \R$ and a $1$-cocycle $u: \Gamma \to \U(\N)$. Denote by $\Psi: \N \rtimes_{\theta} \Gamma \to \N  \weaktensor \R \rtimes_{\theta \otimes \Id} \Gamma$ the $\ast$-isomorphism given by sending $x \rtimes \lambda_g$ to $\phi(x) \, u_g \rtimes \lambda_h$. Let us denote by $\tau_\phi$ the tracial state $(\tau \otimes \tau_\R) \circ \phi$. We have that $\tau_\phi$ is faithful, since both $\tau \otimes \tau_\R$ are, and $\theta$-invariant. We have that the following diagram commute
  \begin{equation*}
    \xymatrix@C=1.75cm@R=1.25cm{
      L^p(\N \rtimes_\theta \Gamma) \ar[d]^{\Id \rtimes T_m} \ar[r]^-{\Psi}
        & L^p(\R \weaktensor \N \rtimes_{\Id \otimes \theta} \Gamma;\tau_\phi \rtimes \tau_\Gamma) \ar[r]^-{j_p} \ar[d]^{\Id_\R \otimes \Id \rtimes T_m}
        & L^p(\R \weaktensor \N \rtimes_{\Id \otimes \theta} \Gamma) \ar[d]^{\Id_\R \otimes \Id \rtimes T_m} \\
      L^p(\N \rtimes_\theta \Gamma) \ar[r]^-{\Psi}
        & L^p(\R \weaktensor \N \rtimes_{\Id \otimes \theta} \Gamma;\tau_\phi \rtimes \tau_\Gamma)  \ar[r]^-{j_p} 
        & L^p(\R \weaktensor \N \rtimes_{\Id \otimes \theta} \Gamma).
    }
  \end{equation*}
  The map $\Psi$ extends to an $L^p$-isometry since it is trace preserving by the definition of $\tau_\phi$. The map $j_p$ is $L^p$-isometric by Lemma \ref{lem:ChangeTrace}. Both intertwine the required multipliers and are equivariant. Therefore the norm of $\Id \rtimes T_m$ equals that of $\Id_\R \otimes(\Id \rtimes T_m)$ and the claim follows.
\end{proof}

The Theorem \ref{thm:McDuffouterconjugate} above can be simplified in certain cases or extended to more general contexts. For instance, if $\N = \R$, then the conditions of strong central triviality can be reduced to freeness. The reason to require freeness for the action $\theta$ is that, if we assume that $\theta_g$ can be inner for some $g \neq e$, then, the it is posible to produce a cohomological obstruction to the outer conjugation. General actions of amenable groups on $\R$ are not all equivalent under outer conjugation and Connes' standard invariant, see \cite{Connes1977Periodic}, is required. Although we have formulated Theorem \ref{thm:McDuffouterconjugate} in the case of finite $\N$ it s also possible to obtain similar results in the semifinite case. 

We also point out the the automatic boundedness of $\Id \rtimes T_m: L^p(\N \rtimes_\theta \Gamma) \to L^p(\N \rtimes_\theta \Gamma)$ for some actions on sufficiently large algebras $\N$ is very different, and a priori unrelated, to the problem of determining if the map $T_m:L^p(\L \Gamma) \to L^p(\L \Gamma)$ is completely bounded. The analogous construction to what we have used in Corollary \ref{cor:AutomaticCB} would be to try to embed arbitrarily large matrices $M_k(\CC) \into M_k(\Gamma) \otimes \L \Gamma \cong \L \Gamma$, when $\L \Gamma$ is McDuff, in a way that intertwines $\Id_{M_k} \otimes T_m$ and $T_m$, which is not doable for all $m \in A(\Gamma)$, since $m = \delta_{g_0} \in A(\Gamma)$ and $T_m$ is of rank $1$, while $\Id_{M_k} \otimes T_m$ would have larger rank. 

\section{Transference and amenable equivariant homomorphisms} \label{sct:AmenableHom}
Let $\M$, $\N$ be von Neumann algebras. We have that a Hilbert space $H$ is a $W^\ast$-correspondence  iff there are normal $\ast$-homomorphisms $r: \N^\op \to \B(H)$ and $\ell: \M \to \B(\H)$, whose ranges commute. Such category of bimodules with bounded bimodular maps as their morphisms was introduced as a suitable substitute for groups representations in the context of von Neumann algebras, see \cite{ConnesJones1985}. More material on correspondences can be found on \cite{Anan1995AmenableCorr, Popa1986Notes} as well as \cite[Appendix B]{ConnesBook1994} and \cite[Chapter 13]{AnaPopa2017II1}. 
The homomorphisms of $\M$ into $\B(\H)$ are usually called the right and left actions respectively and we will denote them simply by
\[
  x \cdot \xi \cdot y = \ell(x) \, r(y) \, \xi,
\]
for every $x \in \N$, $y \in \M$ and $\xi \in H$. We will write $H$ as ${}_\N H_\M$ whenever we want to make the dependence of $\N$ and $\M$ explicit. Like in the context of representations, it is possible to define a natural notion of weak containment of correspondences ${}_\N K_\M \prec {}_\N H_\M$ as follows

\begin{definition}[{\cite[Definition 13.3.8]{AnaPopa2017II1}}]
  \label{def:WeakContainment}
  It is said that $K$ is \emph{weakly contained} in $H$, and denoted ${}_\N K_\M \prec {}_\N H_\M$ iff for every $\varepsilon > 0$ and finite sets $E \subset \N$ and $F \subset \M$, it holds that for every unit vector $\xi$, there are vectors $\{\eta_1, \eta_2, ..., \eta_m\}$ such that
  \[
    \Big| \langle \xi, x \cdot \xi \cdot y \rangle - \sum_{j = 1}^m \langle \eta_j, x \cdot \eta_j \cdot y \rangle \Big| <\varepsilon,
  \]
  for every $x \in E$, $y \in F$.
\end{definition}

The definition above admits the following characterizations, see \cite[Section 13.3]{AnaPopa2017II1}

\begin{proposition}
  \label{prp:WeakContainment}
  Let ${}_\N K_\M$ and ${}_\N H_\M$ be correspondences. The following are equivalent
  \begin{enumerate}[label={\rm \textbf{(\roman*)}}, ref={\rm \textbf{(\roman*)}}]
    \item ${}_\N K_\M \prec {}_\N H_\M$.
    \item There is a net of maps $T_\alpha: K \to H \otimes \ell^2$ such that
    \begin{eqnarray*}
      \big\| T_\alpha( x \cdot \xi \cdot y ) - x \cdot T(\xi) \cdot y \big\|
        & \to & 0\\
      \big\langle T_\alpha(\xi), T_\alpha(\eta) \big\rangle
        & \to & \big\langle \xi, \eta \big\rangle.
    \end{eqnarray*}
    for every $\xi, \eta \in H$ and $x \in \M$ and $y \in \N$.
    \item There is a bimodular isometry
    $T:{}_\N K_\M \tto \leftidx{_\N}{{(H \otimes_2 \ell^2)}}_\M^\U$
  \end{enumerate}
\end{proposition}

Let $\N$ be a semifinite von Neumann algebra with a n.s.f weight $\tau: \M_+ \to [0,\infty]$. Its trivial correspondence $_{\N}L^2(\N)_{\N}$ is given by the GNS representation associated to the tracial weight and the natural left and right actions. This correspondence plays the same role as the trivial representation of groups. We recall also that if $K = L^2(\N)$ is the proposition above, then we can remove the extra $\ell^2$ factor and therefore $_{\N}L^2(\N)_{\N} \prec {}_{\N}H_{\N}$ iff there is a bimodular isometry $J$ from $L^2(\N)$ into $H^\U$. The following follows after taking a sequence of vectors $\xi_\alpha \in \Ball(H)$ such that $J(\1) = (\xi_\alpha)_\alpha^\U$ gives the following.

\begin{proposition}
  \label{prp:WeakContainmentTriv}
  Let ${}_\N H_\M$ be a correspondence. The following are equivalent
  \begin{enumerate}[label={\rm \textbf{(\roman*)}}, ref={\rm \textbf{(\roman*)}}]
    \item \label{itm:WeakContainmentTriv.1}
    $_{\N}L^2(\N)_{\N} \prec {}_{\N}H_{\N}$.
    \item \label{itm:WeakContainmentTriv.2}
    There is a sequence of unit vectors $(\xi_\alpha)_\alpha$ in $H$ such that:
    \begin{itemize}
      \item $\displaystyle{\lim_\alpha \big\| x \cdot \xi_\alpha - \xi_\alpha \cdot x \big\|_2 = 0}$, for every $x \in \N$.
      \item $\displaystyle{ \lim_{\alpha} \big\langle \xi_\alpha, \, x \cdot \xi_\alpha \big\rangle = \tau(x)}$, for every $x \in \N$.
    \end{itemize}     
  \end{enumerate}
\end{proposition}

\begin{remark} \normalfont
  \label{rmk:FactorCase}
  Recall that the second condition on Proposition \ref{prp:WeakContainmentTriv}.\ref{itm:Amenable1Cocycles.2} is superfluous if $\N$ is factor. Indeed, by the centrality of the vectors $\xi_\alpha$, we have that any weak-$\ast$ accumulation point $\varphi:\N \to \CC$ of the states
  \[
    \varphi_\alpha(x)
    \, = \,
    \big\langle \xi_\alpha, \pi(x) \, \xi_\alpha \big\rangle
  \] 
  is a (non-necessarily normal) tracial state on $\N$. But the uniqueness of the trace of $\N$ gives that $\tau = \varphi$.   
\end{remark}

\textbf{GNS construction for completely positive maps.}
Recall that given a normal and completely positive map $\Phi: \N \to \M$, where we are are assuming the algebra $\M$ to be semifinite and have a faithful tracial weight $\tau$, there is an associated $\N$-$\M$-correspondence $H(\Phi)$. Let $\mathfrak{N}_\tau \subset \M$ be the dense subspace given by $\M \cap L^2(\M)$. Let $\N \algtensor \mathfrak{N}_\tau$ be the algebraic tensor product of $\N$ and $\mathfrak{N}_\tau$ and define the positive sesquilinear form 
\[
  \bigg\langle \sum_{j = 1}^m x_1 \otimes y_2, \sum_{k = 1}^n x_2 \otimes y_2 \bigg\rangle_\Phi
  = 
  \tau \bigg\{ \sum_{j = 1}^m \sum_{k = 1}^n y_1^\ast \, \Phi ( x_1^\ast x_1 ) \,  y_1 \bigg\},
\]
Quotienting out the nulspace $N$ of the sesquilinear form above gives a pre-inner product and completing with respect to the associated norm gives the Hilbert space $H(\Phi)$. The left and right actions are defined by extension of
\[
  x \cdot (z \otimes w) \cdot y = x \, z \otimes w \, y,
\]
and trivial computations gives that both are bounded and normal representations.

If $H$ is a $\N$-$\M$-correspondence its \emph{contragradient} $\bar{H}$ is the $\M$-$\N$ correspondence given by the conjugate Hilbert space $\bar{H}$ together with the actions
\[
  x \cdot \bar{\xi} \cdot y \, = \, \overline{ y^\ast \cdot \xi \cdot x^\ast},
\]
for every $\xi \in H$, $x \in \M$ and $y \in \N$.

\textbf{Connes' tensor product.}
The last notion that we need to recall in order to define what amenable completely positive maps is that of Connes' tensor product. If $H$ and $K$ are $\N$-$\M$ and $\M$-$\Q$ correspondences respectively, for semifinite von Neumann algebras $\N$, $\M$ and $\Q$, the Connes' tensor product is a third $\N$-$\Q$ correspondence $H \weaktensor_{\M} K$. Let $H^\circ \subset H$ be the set of all \emph{left-bounded} vectors of $H$, ie vectors for which the operator $x \mapsto \xi \cdot x$ extends to a bounded operator $L_\xi: L^2(\M) \to H$. Given to left-bounded vectors $\xi, \eta \in H^\circ$ their $\M$-product is given by
\[
  \langle \xi, \eta \rangle_\M = L_\xi^\ast L_\eta.
\]
Observe that $L_\xi^\ast L_\eta: L^2(\M) \to L^2(\M)$ is right $\M$-modular and therefore belongs to $\M \subset \B(L^2 \M)$. We define $H \weaktensor_\M K$ as the space resulting from taking $H^\circ \algtensor K$ with the sesquilinear positive form given by linear extension of
\begin{equation}
  \label{eq:InnerprodConnesFusion}
  \langle \xi_1 \otimes \eta_1, \xi_2 \otimes \eta_2 \rangle 
  \, = \,
  \big\langle \eta_1, \langle \xi_1, \xi_2 \rangle_\M \, \eta_2 \big\rangle_K.
\end{equation}
Quotienting out the nulspace associated to that form and taking the metric closure gives the Connes' product. Slightly ambiguously, we will represent the class modulo the nulspace associated to $\xi \otimes \eta$ by $\xi \otimes \eta$ itself. It is easily seen that for every $x \in \M$ it holds that
\[
  \xi \otimes x \cdot \eta - \xi \cdot x \otimes \eta \, = \, 0,
\]
where $\xi \in H^\circ$ and $\eta \in K$.

The following notion of amenable correspondence has its origin in \cite{Popa1986Notes} and \cite{Anan1995AmenableCorr}. It is also a parallel to the notion of amenable representation \cite{Bekka1990}.

\begin{definition}[{\cite{Anan1995AmenableCorr}}] \
  \begin{enumerate}[label={\rm \textbf{(\roman*)}}, ref={\rm (\roman*)}]
    \item 
    A correspondence $_{\N}H_{\M}$ is (left) amenable iff $L^2(\N) \prec H \otimes_{\M} \overline{H}$.
    \item 
    A complete positive map $\Phi: \N \to \M$ is left amenable iff its associated $\N$-$\M$ correspondence $H(\Phi)$ is.
  \end{enumerate}   
\end{definition}

The case in which the completely positive map above is a normal $\ast$-homorphism between semifinite von Neumann algebras $\pi: \N \to \M$ is quite illustrative and easy to describe. First start noticing that $H(\pi) \cong L^2(\M)$ as $\N$-$\M$-bimodules, where the left and right actions on $L^2(\M)$ are given by
\[
  x \cdot \xi \cdot y \, = \, \pi(x) \, \xi \, y.
\]
The bimodular isomorphism $W: H(\pi) \to L^2(\M)$ is just defined over $\N \algtensor \M$ as 
\[
  \sum_{j = 1}^n x_j \otimes y_j \xmapsto{\quad W \quad} \sum_{j = 1}^n \pi(x_j) \, y_j
\]
and a straightforward computation yields that the above map is isometric and bimodular, the surjectivity is immediate.
Next, notice that
\[
  H(\pi) \, \weaktensor_\M \, \overline{H(\pi)} \, \cong \, L^2(\M)
\]
is another isomorphism of $\N$-$\N$-bimodules, where the actions on the right hand side are given by $x \cdot \xi \cdot y \, = \, \pi(x) \xi \pi(y)$. In order to construct that isomorphism just notice that $L^2(\M)^\circ = L^2(\M) \cap \M$ and that, for left bounded vectors $\xi, \eta \in L^2(\M) \cap \M$ their $\M$-valued inner product is $\langle \xi, \eta \rangle_\M = \xi^\ast \eta$. Then \eqref{eq:InnerprodConnesFusion} gives that $\xi \otimes \eta \mapsto \xi \, \eta$ is a bimodular isomorphism. As a consequence of this discussion we have that $\pi$ is (left) amenable iff
\[
  {}_\N L^2(\N)_\N \prec {}_{\pi[\N]} L^2(\M)_{\pi[\N]},
\]
or equivalently iff there is a $\N$-bimodular isometry $L^2(\N) \into L^2(\M)^\U$, for some proper ultrafilter $\U$. Observe that, in this particular case, we can choose the sequence of unit vectors $(\xi_\alpha)_\alpha \subset L^2(\M)$ that satisfy the properties of Proposition \ref{prp:WeakContainmentTriv}.\ref{itm:WeakContainmentTriv.2} to be in the positive cone $L^2(\M)_+ \subset L^2(\M)$ without loss of generality. Indeed, let us notice that if $\xi_\alpha$ is a centralizing sequence, meaning that $(\xi_\alpha)_\alpha^\U \in L^2(\M)^\U \cap \pi[\N]'$, so is $\xi_\alpha^\ast$ and therefore $\xi_\alpha \, \xi_\alpha^\ast$ is again a centralizing sequence in $L^1(\M)$, i.e. it lays in $L^1(\M)^\U \cap \pi[\N]'$. But by Lemma \ref{lem:Approximate}, we have that 
\begin{equation}
  \label{eq:PositiveSequence}
  \eta_\alpha = |\xi_\alpha^\ast| = (\xi_\alpha \, \xi_\alpha^\ast)^\frac12
\end{equation}
is again a sequence of centralizing unit vectors in $L^2(\M)_+$. To see that they also satisfy the trace condition, we just check that
\[
  \tau_\N(x)
  \, = \, \lim_{\alpha} \big\langle \xi_\alpha, \pi(x) \xi_\alpha \big\rangle
  \, = \, \lim_{\alpha} \tau_\M \big( \xi_\alpha \, \xi^\ast_\alpha \, \pi(x) \big)
  \, = \, \lim_{\alpha} \big\langle |\xi_\alpha^\ast|, \pi(x) \, |\xi_\alpha^\ast| \big\rangle
  \, = \, \lim_{\alpha} \big\langle \eta_\alpha, \pi(x) \, \eta_\alpha \big\rangle.
\]
 

The following proposition establishes a connection between the amenability of an embedding $\pi: \L \Gamma \into L^\infty(\Omega) \rtimes_\theta \Gamma$ and the transference results on Theorem \ref{thm:IsometricInclusion}.

\begin{theorem}
  \label{thm:Extrpolation}
  Let $\pi: \L G \to L^\infty(\Omega) \rtimes_\theta \Gamma$ be a normal $\ast$-homomorphism. The following are equivalent
  \begin{enumerate}[label={\rm \textbf{(\roman*)}}, ref={\rm (\roman*)}]
    \item \label{itm:Extrapolation.1}
    $\pi$ is amenable.
    \item \label{itm:Extrapolation.2}
    For every $1 \leq p < \infty$, there is an completely isometric map
    \[
      L^p(\L \Gamma) \xrightarrow{\quad J_p \quad} \prod_{\U} L^p(\Omega \rtimes_\theta \Gamma),
    \]
    that is also bimodular, i.e. $J_p(x \, \varphi \, y ) = \pi(x) \, J_p(\varphi) \, \pi(y)$.
  \end{enumerate}
  Furthermore if, $\pi$ is equivariant and the isometric and bimodular map $L^2(\L \Gamma) \into L^2(\Omega \rtimes_\theta \Gamma)$ of point \ref{itm:Extrapolation.1} intertwines $T_m$ and $\Id \rtimes T_m$, for every $m \in B(\Gamma)$, the Fourier-Stieltjes algebra of $\Gamma$, so does every other $L^p$ isometry $J_p$ in point \ref{itm:Extrapolation.2}.
\end{theorem}

\begin{proof}[Proof (of Theorem \ref{thm:Extrpolation})]
  For the proof of \ref{itm:Extrapolation.1} $\implies$ \ref{itm:Extrapolation.2}, we start by noticing that, since $\pi$ is amenable there is a sequence of vectors $(\xi_\alpha)_\alpha \subset L^2(\Omega \rtimes_\theta \Gamma)$ satisfying the properties of Proposition \ref{prp:WeakContainmentTriv}\ref{itm:WeakContainmentTriv.2}, that is, $\xi_\alpha$ being asymptotically central and satisfying that the vector states $\omega_{\xi_\alpha, \xi_\alpha}(x) = \langle \xi_\alpha, x \, \xi_\alpha \rangle$ converge on the image of $\pi$ to the canonical trace of $\L \Gamma$. But then, we can define the family of maps $J_p^\alpha:L^p(\L \Gamma) \to L^p(\Omega \rtimes_\theta \Gamma)$ as
  \[
    J^\alpha_p(x) = u_\alpha \, |\xi_\alpha|^\frac1{p} \, \iota(x) \, |\xi_\alpha|^\frac1{p},
  \]
  where $\xi_\alpha = u_\alpha \, |\xi_\alpha|$ is the polar decomposition of the vector $\xi_\alpha \in L^\infty(\Omega \rtimes_\theta \Gamma)$. Then, proceeding like in the proof of Theorem \ref{thm:IsometricInclusion} we can obtain easily that
  \begin{enumerate}[label={\rm \textbf{(\arabic*)}}, ref={\rm \textbf{(\arabic*)}}]
    \item \label{itm:ExtrpolationProof.1}
    $\displaystyle{\big\| J_p^\alpha: L^p(\L \Gamma) \to L^p(\Omega \rtimes_\theta \Gamma) \big\| \leq 1}.$
    \item \label{itm:ExtrpolationProof.2}
    $\displaystyle{\lim_{\alpha} \big\langle J_p^\alpha(x), J_p^\alpha(y) \big\rangle = \langle x, y \rangle}$ for every $x \in L^p(\L \Gamma)$ and $y \in L^q(\L \Gamma)$ with $\frac1{p} + \frac1{q} = 1$.
  \end{enumerate}
  Therefore, the ultraproduct map $J_p = (J_p^\alpha)^\U$ is isometric. The fact that $\xi_\alpha$ is asymptotically central in $L^2(\Omega \rtimes_\theta \Gamma)$ gives, by Lemma \ref{lem:Approximate}, that the sequences $u_\alpha \, |\xi_\alpha|^\frac1{p}$ and $|\xi_\alpha|^\frac1{p}$ are asymptotically central in $L^{2p}(\Omega \rtimes_\theta \Gamma)$. This readily gives that the map $J_p$ is $\L \Gamma$-bimodular in the sense that $J_p(x \, \varphi \, y) = \pi(x) \, J_p(\varphi) \, \pi(y)$. Now, we have to show that, if $J_2$ satisfies the intertwining identity
  \begin{equation}
    \label{eq:IntertwiningJL2}
    J_2 \circ T_m \, = \, (\Id \rtimes T_m)^\U \circ J_2,
  \end{equation}
  then, so does every other $J_p$. Notice that, since every $m \in B(\Gamma)$ is a combination of four positive type functions, we can assume that $m$ is of positive type without loss of generality. Similarly, by the continuity of the map $J_p$ over $L^p(\L \Gamma)$, we have that its is enough to prove the intertwining identity over a dense subset of $L^p(\L \Gamma)$. We will choose $\CC[\Gamma] \subset L^p(\L \Gamma)$. By linearity it is enough to prove that, for every $h \in \Gamma$, it holds that
  \begin{eqnarray*}
     0 & = & \big\| \big(J_p \circ T_m\big)(x) - (\Id \rtimes T_m)^\U \circ J_p(x) \big\|_p\\
       & = & \lim_{\alpha \to \U} 
             \Big\| u_\alpha \, |\xi_\alpha|^\frac1{p} \, \pi \big( T_m x \big) \, |\xi_\alpha|^\frac1{p} 
                    - (\Id \rtimes T_m) \big\{ u_\alpha \, |\xi_\alpha|^\frac1{p} \, \pi(x) \, |\xi_\alpha|^\frac1{p} \big\} \Big\|_p,
  \end{eqnarray*}
  where $x = \lambda_h$. Using the fact that $\pi \circ T_m = (\Id \rtimes T_m) \circ \pi$ by the fact that $\pi$ is equivariant and the the fact that $|\xi_\alpha|^\frac1{p} \in L^\frac{2}{p}(\M)$ asymptotically centralizes $\pi(x)$, we obtain that the above expression is equal to
  \begin{eqnarray*}
    \big\| \big(J_p \circ T_m \big)(x) - (\Id \rtimes T_m) \circ J_p(x) \big\|_p
       & = & \lim_{\alpha \to \U} 
             \Big\| \, m(h) \, u_\alpha \, |\xi_\alpha|^\frac{2}{p} \, \pi(\lambda_h) 
             - (\Id \rtimes T_m) \big\{ u_\alpha \, |\xi_\alpha|^\frac{2}{p} \, \pi(x) \big\} \Big\|_p.
  \end{eqnarray*}
  The following modularity property of $(\Id \rtimes T_m)$
  \begin{equation}
    \label{eq:ModularityProperty}
    (\Id \rtimes T_m)(\xi \, \pi(\lambda_h))
    \, = \,
    (\Id \rtimes T_{\rho_h m})(\xi) \, \pi(\lambda_h), 
  \end{equation}
  where $\rho_h m(g) = m(g \, h)$. Therefore, we only have to verify the claim just for $h = e$ and every positive type $m$, since $\rho_h m$ is of positive type whenever $m$ is. Similarly, by linearity, we can take $m(g) = m(e)^{-1} \, m(g)$ and assume that $m(e) = 1$. gathering all the information together, the claim follows by showing that
  \[
    \lim_{\alpha \to \U} 
    \Big\|  u_\alpha \, |\xi_\alpha|^\frac{2}{p} - (\Id \rtimes T_m) \big\{ u_\alpha \, |\xi_\alpha|^\frac{2}{p} \big\} \Big\|_p 
    \, = \,
    0,
  \]
  for every $m$ of positive type with $m(e) = 1$. Since $\Id \rtimes T_m$ is a (completely) positive operator over $L^\infty(\Omega \rtimes_\theta \Gamma)$ we can apply the Theorem \ref{thm:almostMultiplicativemap} on almost multiplicative maps to prove the result.
\end{proof}

Observe that the above proposition has to be read as an extrapolation argument, by which the isometric and bimodular map defined in the $L^2$-level
\[
  L^2(\L \Gamma) \xhookrightarrow{ \quad J \quad } L^2(\Omega \rtimes_\theta \Gamma)
\]
can be extended to the rest of the $L^p$-scale, with $1 \leq p < \infty$ as a complete isometry preserving both the $\L \Gamma$-bimodular properties as well as the intertwining properties for the operators $T_m$ and $(\Id \rtimes T_m)^\U$. This is something that is rarely possible when working with general Hilbert spaces. Indeed, by Lamperti's theorem \cite{Lamperti1958}, see \cite{Yeadon1981Lamperti} as well for a noncommutative analogue, there are many $L^2$-isometries that to not extend to $L^p$. Nevertheless, it seems that turning the $L^2$-space into a bimodule gives you enough structure to have a natural candidate to the associated $L^p$-space. Indeed, if ${}_{\N}H_\N$ is a bimodule, a $L^p$ extrapolation would be any noncommutative $L^p$-bimodule $\X$, in the sense of \cite{JungeSherman2005}, such that it will have a dense subset (its $L^2$-part) isomorphic as a bimodule to a dense subset of $H$.

Observe also that the map $\pi: \L \Gamma \into L^\infty(\Omega) \rtimes_\theta \Gamma$ intertwines $T_m$ and $\Id \rtimes T_m$ when $\pi$ is equivariant, and thus of the form described in Theorem \ref{thm:EquivariantMap}.\ref{itm:EquivariantMap1}. We will next give a complete assessment of the amenable and equivariant $\ast$-homomorphisms $\pi: \L \Gamma \to L^\infty(\Omega) \rtimes_\theta \Gamma$ as well as those amenable $\pi$ such that their associated isometric map $L^2(\L \Gamma) \into L^2(\Omega \rtimes_\theta \Gamma)^\U$ intertwine Fourier multipliers. Before that, we will say that given a $L^\infty(\Omega)$-valued symbol $(m_g)_g \in \ell^\infty[L^\infty(\Omega)] = L^\infty(\Gamma \times \Omega)$ its associated \emph{$L^\infty$-valued Fourier multiplier} its the operator $T_m$ acting on $L^\infty(\Omega) \rtimes_\theta \Gamma$ given by
\[
  T_m \bigg( \sum_{g \in \Gamma} f_g(\omega) \rtimes \lambda_g \bigg)
  \, = \,
  \sum_{g \in \Gamma} m_g(\omega) \, f_g(\omega) \rtimes \lambda_g.
\]
Observe that, like in the case of $L^\infty(\Omega)$-valued Herz-Schur multipliers, there is a bit of ambiguity in denoting both the usual Fourier multipliers $\Id \rtimes T_m$ and the $L^\infty(\Omega)$-valued ones by the $T_m$.

\begin{theorem}
  \label{thm:AmenableEquivariant}
  Let $\theta: \Gamma \to \Aut(\Omega, \mu)$ be an action, $\kappa: \Gamma \to L^\infty(\Omega;\TT)$ a multiplicative $1$-cocycle and $\pi: \L \Gamma \to L^\infty(\Omega) \rtimes_\theta \Gamma$ the associated $\ast$-homomorphism given by
  \[
    \pi \bigg( \sum_{g \in \Gamma} f_g \rtimes \lambda_g \bigg) \, = \, \sum_{g \in \Gamma} f_g \, \kappa_g \rtimes \lambda_g.
  \]
  Then, we have that
  \begin{enumerate}[label={\rm \textbf{(\roman*)}}, ref={\rm (\roman*)}]
    \item \label{itm:AmenableEquivariant.1} $\pi$ is amenable iff there is a sequence of unit vectors $(\xi_\alpha)_\alpha \subset L^2(\Omega \times \Gamma)$ such that 
    \begin{itemize}
      \item 
      $\displaystyle{
        \lim_{\alpha} \big\| \big( \theta_h \otimes \Ad_h \big)(\xi_\alpha) - M_{\varpi_h}(\xi_\alpha) \big\|_2 = 0
      }$, for every $h \in \Gamma$, where $\varphi_h \in L^\infty(\Gamma \times \Omega)$ is given by $\varpi_h(g,\omega) = \overline{\kappa}_h(\omega) \, \theta_g(\kappa_h)$ and $M_{\varpi_h}$ is its associated pointwise multiplication operator.
      \item 
      $\displaystyle{
        \int_\Omega \kappa_h \bigg( \sum_{g \in \Gamma} \xi_\alpha(g , h) \, \theta_h(\xi_\alpha(g)) \bigg) \, d \mu
        \, = \,
        \delta_{\{h = e\}}
      }$
      for every $h \in \Gamma$.
    \end{itemize} 
    \item \label{itm:AmenableEquivariant.2}
    $\pi$ is amenable and the associated bimodular isometry
    \[
      \leftidx{_{\L \Gamma}}{{L^2(\L \Gamma)}}{_{\L \Gamma}} \tto \leftidx{_{\L \Gamma}}{{L^2(\Omega \rtimes_\theta \Gamma)}}{_{\L \Gamma}}^\U
    \]
    intertwines $T_m$ and $\Id \rtimes T_m$, for every $m \in B(\Gamma)$, iff $\Omega$ has a $\Gamma$-invariant mean.
  \end{enumerate}
\end{theorem}

\begin{proof}
  Point \ref{itm:AmenableEquivariant.1} follows after a routine application of Proposition \ref{prp:WeakContainmentTriv}. Indeed, if $\pi$ is amenable there is a sequence of unit vectors $(\xi_\alpha)_\alpha \subset L^2(\Omega \rtimes_\theta \Gamma)$ that is centralizing, or equivalently it is asymptotically invariant by the action of $\Gamma$ given by $\xi \mapsto \pi(\lambda_h) \, \xi \, \pi(\lambda_h)^\ast$. Therefore, we have that
  \[
    \begin{split}
      (\kappa_h \rtimes \lambda_h ) \, \bigg( \sum_{g \in \Gamma} & \, \xi_\alpha(g) \rtimes \lambda_g \bigg) \, (\theta_{h^{-1}}(\overline{\kappa}_h) \rtimes \lambda_{h^{-1}} )\\
      & = \, \sum_{g \in \Gamma} \kappa_h \, \theta_h(\xi_\alpha(g)) \, \theta_{h g h^{-1}}(\overline{\kappa}_h) \rtimes \lambda_{h \, g h^{-1}}  \\
      & = \, \sum_{g \in \Gamma} \kappa_h \, \theta_h(\xi_\alpha(h^{-1} g \, h)) \, \theta_{g}(\overline{\kappa}_h) \rtimes \lambda_{g}
      \, = \, \sum_{g \in \Gamma} \overline{\varpi_h(g, \omega)} \, \theta_h(\xi_\alpha(h^{-1} g \, h))(\omega) \rtimes \lambda_{g},
    \end{split}
  \]
  where $\xi_\alpha = \sum_{g} \xi_\alpha(g) \rtimes \lambda_g$ and each of the coefficients $\xi_\alpha(g)$ is uniquely determined by $\EE[\xi_\alpha \, \lambda_g^\ast]$. Using the Plancherel identity and the fact that $\mu$ is $\theta$-invariant, we have that $L^2(\Omega \rtimes_\theta \Gamma) \cong L^2(\Omega \times \Gamma)$, where the isomorphism is given by sending $\xi_\alpha(g) \rtimes \lambda_g$ to $\xi_\alpha(g) \otimes \delta_g \in L^2(\Omega) \otimes_2 \ell^2(\Gamma)$. Therefore, the above expression is at diminish distance of $\xi_\alpha$ iff the net $(\xi_\alpha)_\alpha \subset L^2(\Omega \times \Gamma)$ satisfies that
  \[
    \lim_{\alpha} \big\| \big(\theta_h \otimes \Ad_h \big) \xi_\alpha - M_{\varpi_h} \xi_\alpha \big\|_2 = 0,
  \]
  or every $h \in \Gamma$. The second condition follows by imposing that $\tau( \xi_\alpha^\ast \pi(\lambda_h) \, \xi_\alpha) = \delta_{\{h = e\}}$.
  
  For point \ref{itm:AmenableEquivariant.2} we need to see that if $J_2^\alpha(x) = \xi_\alpha \, \pi(x)$ intertwines $\Id \rtimes T_m$ and $T_m$, then without loss of generality we can assume that $\xi_\alpha \in L^2(\Omega) \otimes_2 \CC \1 \subset L^2(\Omega \rtimes \Gamma)$. Arguing like in the proof of \ref{thm:Extrpolation} we have that $J_2 \circ T_m = (\Id \rtimes T_m) \circ J_2$ iff for every $h \in \Gamma$
  \begin{eqnarray*}
    0 \, = \, 
    \lim_{\alpha \to \U} \Big\| \xi_\alpha \, \pi \big( T_m \lambda_h \big) - (\Id \rtimes T_m)(\xi_\alpha \, \pi(\lambda_h) \Big\|_2
      & = & \lim_{\alpha \to \U} \Big\| \xi_\alpha \, \pi \big( T_m \lambda_h \big) - (\Id \rtimes T_m)\big(\xi_\alpha \, \pi(\lambda_h) \big)\Big\|_2\\
      & = & \lim_{\alpha \to \U} \Big\| m(h) \, \xi_\alpha - (\Id \rtimes T_{\rho_h m})(\xi_\alpha) \Big\|_2.
  \end{eqnarray*}
  We have used the fact that $\pi$ in equivariant, the modularity property \eqref{eq:ModularityProperty} and the fact that the $L^2$-norm is invariant by multiplication by the unitary $\pi(\lambda_h) = \kappa_h \rtimes \lambda_h$. But now, taking $m = \delta_e$ implies that the sequences $\xi_\alpha$ and $T_{\delta_e}(\xi_\alpha) = \xi_\alpha(e) \rtimes \lambda_e \in L^2(\Omega) \subset L^2(\Omega \rtimes_\theta \Gamma)$ induce the same element in the ulptrapower and therefore we can assume without loss of generality that $\xi_\alpha$ is equal to $\xi_\alpha(e) \rtimes \lambda_e$. The condition \ref{itm:AmenableEquivariant.1} over vectors in $L^2(\Omega) \subset L^2(\Omega \rtimes_\theta \Gamma)$ implies that the vectors are asymptotically $\Gamma$ invariant and therefore $\Omega$ has a $\Gamma$-invariant mean. 
\end{proof}

\begin{remark} \normalfont
  Several remarks are in order. The first is that in point \ref{itm:AmenableEquivariant.1} the second condition is superfluous whenever $\Gamma$ is i.c.c, see Remark \ref{rmk:FactorCase}.
  The next is that the theorem above has to be understood as a negative result stating that the transference technique used in Theorem \ref{thm:IsometricInclusion} can not be applied in more general context by changing the natural inclusion $\L \Gamma \subset L^\infty(\Omega) \rtimes_\theta \Gamma$ by a more general injective, normal and equivariant $\ast$-homomorphism $\pi: \L \Gamma \into L^\infty(\Omega) \rtimes_\theta \Gamma$. Indeed, both Proposition \ref{thm:Extrpolation} and Theorem \ref{thm:AmenableEquivariant} above imply that such technique requires the existence of a $\Gamma$-invariant mean on $\Omega$.
  Nevertheless, it is possible that, provided there are actions without a $\Gamma$-invariant mean such that the condition \ref{itm:AmenableEquivariant.1} holds, that the map $J_2$ may intertwine a subfamily of multiplies $m$. In particular, assume that there is an infinite conjugacy class $\C \subset \Gamma$ such that all of the vectors $\xi_\alpha$ on \ref{itm:AmenableEquivariant.1} are supported in $L^2(\Omega \times \C)$ and assume that $\Gamma_0 \subset \Gamma$ is a proper subgroup generated by $\C \subset \Gamma$. Then, $J_2$ intertwines every left $\Gamma_0$-invariant multiplier $m \in \B(\Gamma)$ and by the techniques on \ref{thm:Extrpolation} $J_p$ intertwines the same family of multipliers at the same time. Assuming that Zimmer-amenable actions satisfying \ref{itm:AmenableEquivariant.1} could be found, that would gives families of multipliers for which the Schur and Fourier norms are equal, thus obtaining a positive solution of a weakening of Problem \ref{prb:TrnsferenceActions}.
\end{remark}

\section{Transference and amenable equivariant cp maps} \label{sct:AmenableCP}
In this last section we will explore in which situation the completely positive and equivariant maps of Theorem \ref{thm:EquivariantMap}.\ref{itm:EquivariantMap2} $\Phi: \L \Gamma \to L^\infty(\Omega) \rtimes_\theta \Gamma$ given by extension of
\[
  \lambda_g \xmapsto{\quad \Phi \quad} \varphi_g \rtimes \lambda_g,
\]
where $\varphi_g(\omega) = \langle \xi, \kappa_g(\omega) \xi \rangle$ is given by the matrix coefficients of a multiplicative $1$-cocycle $\kappa: \Gamma \to L^\infty(\Omega; \U(H))$ for some Hilbert space $H$. We will also denote by $\pi$ the normal $\ast$-homomorphism
\[
  \L \Gamma \xrightarrow{\quad \pi \quad}
  \B(H) \weaktensor \big( L^\infty(\Omega) \rtimes_\theta \Gamma \big)
  \cong L^\infty(\Omega;\B(H)) \rtimes_{\theta \otimes \Id} \Gamma
\]
given by linear extension of $\pi(\lambda_g) = \kappa_g(\omega) \rtimes \lambda_g$.

We introduce the following amenability condition for multiplicative $1$-cocycles.

\begin{definition}
  \label{def:Amenable1Cocycles}
  Let $\theta: \Gamma \to \Aut(\Omega, \mu)$ be an action and $\kappa: \Gamma \to  L^\infty(\Omega;\U(H))$ a multiplicative $1$-cocyle for some Hilbert space $H$. We will say that $\kappa$ is \emph{amenable} iff either of the following equivalent conditions hold
  \begin{enumerate}[label={\rm \textbf{(\roman*)}}, ref={\rm (\roman*)}]
    \item \label{itm:Amenable1Cocycles.1} 
    There is a (not necessarily normal) state $m:L^\infty(\Omega;\B(H)) \to \CC$ such that
    \[
      m \big( (\theta_h \otimes \Id) x \big) = m \big( \Ad_{\kappa_{h^{-1}}} x \big),
    \]
    for every $x \in L^\infty(\Omega;\B(H))$ and $h \in \Gamma$.
    \item \label{itm:Amenable1Cocycles.2} 
    There is a net of unit vectors $(\xi_\alpha)_\alpha \subset L^2(\Omega;S^2(H))$, 
    where $S^2(H)$ are the Hilbert-Schmidt operators on $H$, such that
    \[
      \lim_{\alpha} \big\| \Ad_{\kappa_h} \circ (\theta_h \otimes \Id)(\xi_\alpha) - \xi_\alpha \big\|_2 = 0.
    \]
  \end{enumerate}   
\end{definition}

\begin{remark} \normalfont
  Observe that, by doing the change of variable $x \mapsto \Ad_{\kappa_{h^{-1}}}^{-1} x$ we obtain that the state $m$ on point \ref{itm:Amenable1Cocycles.1} is invariant under the action $x \mapsto (\theta_h \otimes \Id) \circ \Ad_{\kappa_{h^{-1}}}^{-1} x$. But such action of $\Gamma$ is given precisely by
  \[
    (\theta_h \otimes \Id) \circ \Ad_{\kappa_{h^{-1}}}^{-1} = \Ad_{\kappa_h} \circ (\theta_h \otimes \Id).
  \]
  Let us denote by $\theta \rtimes \kappa: \Gamma \to \Aut(L^\infty(\Omega;\, \B(H)))$ the action above. Clearly, it satisfies that its restriction to $L^\infty(\Omega)$ recovers $\theta$. Therefore, we can think of it as an extended action. Condition \ref{itm:Amenable1Cocycles.1} is thus equivalent to the existence of a $\theta \rtimes \kappa$-invariant state. 
\end{remark}

The equivalence between the two notions above follows easily. The fact that \ref{itm:AmenableEquivariant.2} implies \ref{itm:Amenable1Cocycles.1} follows by an application of Banach-Alaouglu theorem. Take the vector states $\varphi_\alpha(x) = \omega_{\xi_\alpha, \xi_\alpha}(x) = \langle \xi_\alpha, x \xi_\alpha \rangle$, we have that any weak-$\ast$ accumulation point is an invariant state $m$. For the reverse implication we need to use Goldstine's theorem. Indeed, given a state $m \in \Ball(L^1(\Omega; \, S^1(H))^{\ast \ast})$ we can approximate by states in $\omega_\alpha = L^1(\Omega; \, S^1(H))$ in the weak-$\ast$ topology. But, taking square roots gives elements in $\xi_\alpha \in L^2(\Omega; S^2)$. A routine application of Hanh-Banach theorem, analogous to that on \cite[Theorem 2.5.11]{BroO2008} gives that the $\xi_\alpha$ satisfy the condition in \ref{itm:Amenable1Cocycles.2}.

We can describe explicitly the correspondences $H(\Phi)$ and $H(\Phi) \, \weaktensor_{L^\infty(\Omega) \rtimes \Gamma} \, H(\Phi)$ as follows.
\begin{proposition}
  \label{prp:CorrepondencesCP}
  Let $\Phi: \L \Gamma \to L^\infty(\Omega) \rtimes_\theta \Gamma$ be an equivariant cp map and let $\kappa: \Gamma \to L^\infty(\Omega; \U(H))$ 
  be its associated multiplicative $1$-cocycle. We have that
  \begin{enumerate}[label={\rm \textbf{(\roman*)}}, ref={\rm (\roman*)}]
    \item \label{itm:CorrespondencesCP.1}
    There is bimodular isomorphism
    \[
      \leftidx{_{\L \Gamma}}{{H(\Phi)}}_{L^\infty(\Omega) \rtimes \Gamma} 
      \, \cong \, 
      \leftidx{_{\L \Gamma}}{{ \big( \, H \otimes_2 L^2(\Omega \rtimes_\theta \Gamma) \, \big)}}_{L^\infty(\Omega) \rtimes \Gamma},
    \]
    where the left and right actions are given by
    \begin{eqnarray*}
      \lambda_h \cdot \xi & = & \pi(\lambda_h) \, \xi \\
      \xi \cdot (f \rtimes \lambda_h) & = & \xi \, (\1 \otimes f \rtimes \lambda_h) 
    \end{eqnarray*}
    respectively.
    \item \label{itm:CorrespondencesCP.2}
    There is a bimodular isomorphism 
    \[
      \leftidx{_{\L \Gamma}}{{ \big( \, H(\Phi) \, \weaktensor_{L^\infty(\Omega) \rtimes \Gamma} \, \overline{H(\Phi)} \, \big)}}{_{\L \Gamma}}
      \, \cong \, 
      \leftidx{_{\L \Gamma}}{{ \big( \, S^2(H) \otimes_2 L^2(\Omega \rtimes_\theta \Gamma) \, \big)}}{_{\L \Gamma}}, 
    \]
    where both the left and right actions are given by the image of $\pi$, 
    i.e. $\lambda_h \cdot \xi \cdot \lambda_g = \pi(\lambda_h) \, \xi \, \pi(\lambda_g)$.  
  \end{enumerate}
\end{proposition}

We will just sketch the proof of the two isomorphisms above. For the first one we will use the fact that the GNS construction for cp maps that associates to them the correspondence $H(\Phi)$ is unique under isomorphisms, see \cite[Chapter 13]{AnaPopa2017II1}. Indeed, if $\N$ and $\M$ are finite von Neumann algebras, given any contractive and completely positive map $\Phi: \N \to \M$, we can associate a pair $(H(\Phi), \xi_0)$, where $_\N H(\Phi)_\M$ is the correspondence described after Proposition \ref{prp:WeakContainmentTriv} and $\xi$ is the bi-cyclic left-bounded vector associated to the class of $\1 \otimes \1$. In this case we can recover $\Phi$ by
\[
  \Phi(x) = L_\xi^\ast \, x \, L_\xi.
\]
We have that in the finite case such construction is unique up to isomorphisms. The subtlety in our case, in which $\M$ is just semifinite, is that $\1 \otimes \1$ doesn't belong to $H(\Phi)$ since $\tau_\M(\1) = \infty$. Nevertheless, it is possible to choose an increasing net of vectors $(\xi_\alpha)_\alpha \subset \M$ of finite trace, whose associated operators $L_\xi: L^2(\M) \to H$ are uniformly bounded and such that $\Phi(x)$ is the strong limit of  $L_{\xi_\alpha}^\ast \, x \, L_{\xi_\alpha}$. Indeed, it is enough to fix $\xi_\alpha$ as the class of $\1 \otimes p_\alpha$, where $p_\alpha$ is an increasing net of finite projections converging in the strong topology to $\1$. It is again true that such property characterizes $H(\Phi)$ as a correspondence. Taking $\xi_\alpha = \xi \otimes p_\alpha$, where $\xi \in H$ satisfies that $\varphi_g = \langle \xi, \kappa_g \, \xi \rangle$ and $p_\alpha = \1_{F_\alpha} \rtimes \lambda_e$ is an increasing net of finite projections where $F_\alpha \subset \Omega$ and $\mu(F_\alpha) < \infty$, gives the desired result. The second isomorphism follows immediately. 

\begin{theorem}
  \label{thm:Amenable1CocyclesTrans}
  Let $\theta: \Gamma \to \Aut(\Omega, \mu)$ be an action $\Phi: \L \Gamma \to L^\infty(\Omega) \rtimes_\theta \Gamma$ an equivariant cp map 
  and $\kappa: \Gamma \to L^\infty(\Omega; \U(H))$ its associated $1$-cocycle.
  \begin{enumerate}[label={\rm \textbf{(\roman*)}}, ref={\rm (\roman*)}]
    \item \label{itm:Amenable1CocyclesTrans.1}
    $\Phi$ is (left) amenable iff there is a sequence of unit vectors 
    $(\xi_\alpha)_\alpha \subset S^2(H) \otimes_2 L^2(\Omega \times \Gamma)$ satisfying that
    \begin{itemize}
      \item If $\Pi: L^2(\Omega; \, S^2(H)) \otimes_2 \ell^2(\Gamma) \to L^2(\Omega; \, S^2(H)) \otimes_2 \ell^2(\Gamma)$ is the operator defined by
      \[
        \Pi(\varphi \otimes \delta_g)
        \, = \,
        \kappa_h^\ast \, \varphi \, \theta_{g}(\kappa_h) \otimes \delta_g,
      \]
      then
      \[
        \lim_\alpha \big\| (\Id_{S^2} \otimes \theta_h \otimes \Ad_{h})(\xi_\alpha) - \Pi(\xi_\alpha) \big\|_2 \, = \, 0,
      \]
      for every $h \in \Gamma$.
      \item It holds that
      \[
        \sum_{g \in \Gamma} \int_\Omega \Tr \big\{ \xi_\alpha(h \, g)^\ast \, \kappa_h \, \theta_h(\xi_\alpha(k)) \big\} \, d \mu = \delta_{\{h = e\}},
      \]
      for every $h \in \Gamma$.
    \end{itemize}
    \item \label{itm:Amenable1CocyclesTrans.2}
    $\Phi$ is (left) amenable and the associated $\L \Gamma$-bimodular isometry 
    \[
      L^2(\L \Gamma) \xrightarrow{\quad J_2 \quad} \prod_{\U} S^2(H) \otimes_2 L^2(\Omega \rtimes_\theta \Gamma)
    \]
    intertwines $T_m$ and $\Id_{S^2} \otimes \Id_{\Omega} \rtimes T_m$, for every $m \in B(\Gamma)$,
    iff $\kappa$ is amenable in the sense of Definition \ref{def:Amenable1Cocycles}.
  \end{enumerate}
\end{theorem}

The proof of the theorem above is exactly like that of Theorem \ref{thm:AmenableEquivariant} land therefore we will omit it. 
The following corollary follows from point \ref{itm:Amenable1CocyclesTrans.2} and the proof of Theorem \ref{thm:Extrpolation}.

\begin{corollary}
  \label{cor:Amenable1CocycleLp}
  Let $\theta: \Gamma \to \Aut(\Omega, \mu)$ be an action $\Phi: \L \Gamma \to L^\infty(\Omega) \rtimes_\theta \Gamma$ an equivariant cp map 
  and $\kappa: \Gamma \to L^\infty(\Omega; \U(H))$ its associated $1$-cocycle. If $\kappa$ is amenable, then 
  for every $1 \leq p < \infty$, there is a complete isometry
  \[
    L^p(\L \Gamma) \xrightarrow{\quad J_p \quad} \prod_{\U} S^p(H) \otimes_p L^p(\Omega \rtimes_\theta \Gamma),
  \]
  which is $\L \Gamma$-bimodular and satisfies that 
  \begin{equation}
    \xymatrix@C=3cm{
      L^p(\L \Gamma) \ar[d]^-{T_m} \ar[r]^-{J_p}
        & \displaystyle{ \prod_{\U} S^p(H) \otimes_p L^p(\Omega \rtimes_\theta \Gamma) \ar[d]^{(\Id \otimes \Id \rtimes T_m)^\U}}\\
      L^p(\L \Gamma) \ar[r]^-{J_p}
        & \displaystyle{ \prod_{\U} S^p(H) \otimes_p L^p(\Omega \rtimes_\theta \Gamma).}
    }
    \label{dia:Amenable1CocycleLp}
  \end{equation}
\end{corollary}

Although the theorem above will give that Corollary \ref{eq:lowerTransferenceBnd} holds whenever the action $\theta$ admits an amenable $1$-cocycle $\kappa$. Sadly, this doesn't provide more exotic examples since restricting the mean $m: L^\infty(\Omega;\B(H)) \to \CC$ to the unital von Neumann subalgebra $L^\infty(\Omega; \CC\1) = L^\infty(\Omega)$ yields a $\Gamma$-invariant mean on $\Omega$. Similarly, in the case in which $\theta$ is Zimmer-amenable and $\kappa$ is amenable, so is $\Gamma$ itself.

\textbf{Acknowledgement.} The author is thankful to Simeng Wang for a personal communication regarding the mistake on \cite{Gon2018CP} and for subsequent discussion on the present manuscript.

\begin{small}
  \bibliographystyle{alpha}
  \bibliography{../bibliography/bibliography}
\end{small}

\vspace{50pt}

\

\hfill \noindent \textbf{Adri\'an M. Gonz\'alez-P\'erez} \\
\null \hfill University Clermont-Auvergne, LMBP \\ 
\null \hfill 3 Place Vasarely 63178 Aubi\'ere, France \\
\null \hfill\texttt{adrian.gonzalezperez@uca.fr}\\
\null \hfill\texttt{chadrian.gzl@gmail.com}
\

\end{document}